\documentclass[11 pt]{article}
\usepackage{style}

\providecommand{\Halmos}{\hfill\ensuremath{\square}}
\providecommand{\Tilde}[1]{\widetilde{#1}}
\providecommand{\Bar}[1]{\overline{#1}}

\title{\LARGE\bfseries A Min-Max Gradient Search Method for Constrained Simulation Optimization}
\author{
Ruiyang Jin\textsuperscript{1}, Siyang Gao\textsuperscript{1}, and Henry Lam\textsuperscript{2}\\
{\small \textsuperscript{1}\textit{Department of Systems Engineering, City University of Hong Kong}}\\
{\small \href{mailto:syeruiyajin@cityu.edu.hk}{\textit{syeruiyajin@cityu.edu.hk}}, \href{mailto:siyangao@cityu.edu.hk}{\textit{siyangao@cityu.edu.hk}}}\\
{\small \textsuperscript{2}\textit{Department of Industrial Engineering and Operations Research, Columbia University}}\\
{\small \href{mailto:khl2114@columbia.edu}{\textit{khl2114@columbia.edu}}}
}
\date{\vspace{-0.4 in}}
\begin{document}
\maketitle

\begin{abstract}
Constrained simulation optimization (CSO) is a general framework for optimizing stochastic systems under performance constraints. It arises widely in practice where objective and constraint evaluations are available only through noisy simulation outputs. Compared with the unconstrained setting, the lack of accessible analytical gradients for simulation-based constraints makes it more challenging to develop efficient solution methods and establish non-asymptotic guarantees. To address this gap, we propose a novel single-loop algorithm, called min-max gradient search (MGS), which integrates a primal-dual framework with stochastic gradient estimators. Unlike conventional stochastic approximation methods based on gradient descent for solving simulation optimization problems, such as \cite{zhou2017gradient} and \cite{hu2025convergence}, MGS performs alternating gradient descent and ascent on the primal and dual variables, which improves the objective while penalizing constraint violations. For the first time, we establish a finite-time convergence guarantee for single-loop CSO algorithms by showing that MGS converges to a stationary solution (a Karush-Kuhn-Tucker point under mild conditions) at a rate of $\tilde{O}(T^{-1/3})$, where $T$ is the number of iterations.  Numerical experiments on a serial queuing system and a 2000-dimensional optimization problem demonstrate the superior performance and scalability of MGS.
\end{abstract}


%%%%%%%%%%%%%%%%%%%%%%%%%%%%%%%%%%%%%%%%%%%%%%%%%%%%%%%%%%%%%%%%%%%%%%

\section{Introduction}\label{sec:Intro}
Simulation optimization (SO) is widely used in operations research to tackle complex problems where analytical models are unavailable or intractable. Motivated by many real-world applications involving multiple performance measures and stochastic constraints, we study the following constrained simulation optimization (CSO) problem:
\begin{equation}
\begin{aligned}\label{eq:constrained_problem}
    \min_{x\in \mathcal{X}}\; & h_0(x)=\mathbb{E}[h_0(x;\xi)],\\ 
    \text{s.t.}\; & h_j(x)=\mathbb{E}[h_j(x;\xi)]\leq 0,\; \forall j\in \mathcal{I},
\end{aligned}    
\end{equation}
where the feasible set $\mathcal{X}\subseteq \mathbb{R}^{d_x}$ is a closed and convex set. The functions $h_j:\mathbb{R}^{d_x}\to\mathbb{R},\forall j\in\{0\}\cup \mathcal{I}$ represent the multiple performance measures of the system under study.  We set $h_0$ as the objective measure and $h_j, \forall j\in\mathcal{I}$ as the constraint measures. In practice, $h_j(x)$ is unknown and can only be estimated using samples $h_j(x; \xi)$, $\forall j \in \{0\} \cup \mathcal{I}$, obtained from a simulation model of the system, where $\xi$ represents the uncertainty in the simulation. 

% For example, in queuing systems, a typical application involves minimizing total service cost by coordinating service rates while maintaining quality \citep{bassamboo2010accuracy}. In inventory management, it is common to minimize overall costs while keeping the stockout rate below a specified threshold \citep{kleijnen2010constrained}. In energy systems, operational costs must often be minimized while ensuring system safety \citep{cordoba2024optimal}.
Real-world problems that fit the form of \eqref{eq:constrained_problem} are prevalent \citep{amaran2016simulation}. Examples include minimizing service costs while maintaining service quality in queuing systems \citep{bassamboo2010accuracy}, minimizing total costs subject to a low stockout rate in inventory management \citep{kleijnen2010constrained}, and minimizing operational costs while guaranteeing safety in energy systems \citep{cordoba2024optimal}. Since these systems are generally too complex to be captured analytically, simulation models provide a promising solution, and the resulting optimization tasks are naturally formulated as CSO problems. However, algorithmic and theoretical developments for CSO remain limited. The difficulty stems from the fact that the objective and constraint functions can only be accessed via noisy simulation outputs. This means that optimization and feasibility checks have to be carried out under uncertainty, which limits the effectiveness of standard search methods and makes it significantly harder to establish finite-time convergence guarantees.

% the inherent system stochasticity and the lack of analytical models for the objective and constraint measures, which make traditional search-based methods ineffective and pose significant challenges for establishing finite-time convergence guarantees for the algorithms.

This research addresses this gap. Our method builds on a primal-dual framework with stochastic gradient estimations by considering the Lagrangian of \eqref{eq:constrained_problem}. Although primal-dual methods are well-studied in nonconvex optimization, they are relatively underexplored in SO \citep{bhatnagar2011stochastic}. By combining the primal-dual framework with stochastic finite-difference gradient estimators, we are able to design an efficient algorithm for CSO in the form of \eqref{eq:constrained_problem} with non-asymptotic guarantees.

% In our approach, the primal variables are updated to improve the objective, while the dual multipliers penalize constraint violations, which effectively balances the feasibility and optimality.

% \textbf{Our Contributions.} Given the lack of efficient methods to address SO with noisy constraints, we leverage a primal-dual framework to alternatively solve a min-max problem defined on the Lagrange function of \eqref{eq:constrained_problem}. 

\textbf{Our Contributions.} Specifically, we employ a Gaussian smoothing-based finite-difference gradient estimator, which requires at least two simulation samples per gradient estimate, and integrate it into a single-loop gradient descent–ascent (GDA) scheme. The primal updates in GDA minimize the objective, while the dual updates maximize the penalty for constraint violations. 
When estimating gradients using finite differences in function values, we introduce common random numbers (CRNs) and analyze how they reduce the variance of gradient estimations.
To stabilize the optimization and accelerate convergence, a regularization term is added to the Lagrangian to introduce strong concavity, and a momentum-based update that combines current and historical gradients is applied to smooth out noisy updates.
% Moreover, a momentum-based update that combines current and historical gradients is applied to stabilize the optimization and accelerate convergence. 
Building on these components, we design the \textit{min-max gradient search} (MGS) algorithm. 

We establish non-asymptotic convergence guarantees for the MGS algorithm. The results show that MGS can converge to a stationary point at a convergence rate of $\tilde{O}(T^{-1/3})$, where $T$ is the number of iterations. Under mild conditions, any such stationary point corresponds to a Karush-Kuhn-Tucker (KKT) solution of \eqref{eq:constrained_problem}. To our knowledge, this work offers the \textit{first} finite-time convergence guarantee for single-loop algorithms to solve the CSO problem in \eqref{eq:constrained_problem}. Additionally, this research is related to the literature on zeroth-order min-max optimization. In this context, our work improves the \textit{best-known} sample complexity bound for single-loop zeroth-order algorithms from $O(d\epsilon^{-4})$ to $\tilde{O}(d\epsilon^{-3.5})$ to find an $\epsilon$-stationary point ($d$ is the dimension of the solution space). The theoretical gains primarily stem from the adoption of strong-concavity regulation and momentum-based updates, whereas CRNs play an important role in empirical improvements.
% \citep{NEURIPS2024_413885e7}.

Finally, the superior performance of MGS is validated through numerical experiments on both a queuing system and a 2000-dimensional stochastic optimization problem.

\textbf{Related Work.} Most literature on SO with constraints considers deterministic constraints on the solutions $x$, which do not involve sampling the constraint measures and are, in essence, closer to unconstrained SO, i.e., $\min_{x\in\mathcal{X}} h_0(x)$ \citep{fu2015handbook}.
In this regard, the classical gradient-based framework of stochastic approximation (SA) handles this problem through iterative gradient estimation combined with projection operators \citep{zhou2017gradient,xu2023gradient,hu2025convergence}. The SA methodology originated with the Robbins-Monro (RM) algorithm \citep{robbins1951stochastic}, which uses direct gradient estimates, and was followed by the Kiefer-Wolfowitz (KW) algorithm \citep{kiefer1952stochastic}, which uses finite-difference gradient approximations. Generally, RM can enjoy the convergence rates of $O(T^{-1/2})$ when unbiased gradient estimations are available, which can also be achieved by KW in certain settings when CRNs are applicable \citep{fu2015handbook, kleinman1999simulation}. 
% can achieve the convergence rate of $O(T^{-1/2})$, while KW typically suffers from an additional bias–variance trade-off induced by the perturbation radius, and has the convergence rate of $O(T^{-1/3})$ \citep{fu2015handbook}.
Although SA algorithms are well-suited for problems with constraint sets that allow for easy projection (such as analytically defined boxes or polyhedra), they are not directly applicable to the constraints in \eqref{eq:constrained_problem} for two reasons: the feasible set induced by $\mathbb{E}[h_j(x;\xi)]\leq 0$ is not available in closed form, and even if its geometry were simple, exact projection would require solving a subproblem with constraints that are only noisily evaluated.
% it is geometrically challenging to do projections for implicit constraints; moreover, the constraints can only be stochastically evaluated, which makes the projections more challenging to properly excute.
% which are defined implicitly through simulation-based mappings.

For CSO problems of the form \eqref{eq:constrained_problem}, one line of work focuses on meta-heuristics \citep{rajwar2023exhaustive}, such as genetic algorithms, which work with a population of candidate solutions and move this population toward better feasible solutions by evaluating multiple performance measures iteratively \citep{udhayakumar2011stochastic,tsai2014genetic,chang2025maximizing}. Although often flexible and easy to implement, such approaches typically lack general theoretical guarantees and can require many simulation evaluations to reliably achieve feasibility in the presence of noise.

The second line of work focuses on penalty-based and primal–dual methods. A widely used strategy in the SO literature is to relax hard constraints by incorporating them as soft penalties or by regularizing the objective function. For example, penalty methods combined with stochastic perturbation or finite-difference gradient estimations have been analyzed for inequality-constrained problems \citep{wang2008stochastic}. Although these techniques can handle functional constraints, the theoretical results have largely focused on asymptotic convergence, and non-asymptotic convergence results remain underdeveloped. Another approach is to work with the Lagrangian of the original constrained problem and reformulate it as a min-max problem \citep{bhatnagar2011stochastic, li2022zeroth, nguyen2023stochastic}. One can then apply GDA-type algorithms to solve it, where the primal update minimizes the objective and the dual update increases the penalty for constraint violations. In particular, \cite{nguyen2023stochastic} established the best-known sample complexity bound of $O(d\epsilon^{-3})$ to find an $\epsilon$-KKT point for problem \eqref{eq:constrained_problem} in the nonconvex setting. However, their result is given for a weighted average of the iterates and relies on a \textit{nested-loop} scheme that calls an inner convex solver at every iteration, which complicates both implementation and parameter tuning. 
In contrast, \textit{single-loop} algorithms are another common algorithmic framework. It avoids solving an inner subproblem at each iteration and thus has a simpler structure, which is more practical to implement. However, there is no finite-time convergence result among this class of algorithms for solving \eqref{eq:constrained_problem}.

Another line of research focuses on surrogate-based methods, mainly referred to as Bayesian optimization (BO) \citep{pmlr-v130-eriksson21a}. They build probabilistic surrogate models for the objective and constraint functions, including constrained expected improvement (CEI) \citep{gardner2014bayesian}, predictive entropy search with constraints \citep{hernandez2016general}, and constrained knowledge gradient (CKG) \citep{ungredda2024bayesian}. Although BO methods can perform well in moderate-dimensional settings, they generally suffer from inefficiency in high-dimensional problems \citep{wang2023recent}. In addition, due to the complexity of the algorithm dynamics, such as CEI and CKG, it is generally difficult to establish non-asymptotic convergence guarantees for BO methods \citep{srinivas2010gaussian,ungredda2024bayesian}.

% Existing research on CSO in the form of \eqref{eq:constrained_problem} is still limited. One line of work focuses on meta-heuristics \citep{rajwar2023exhaustive}, such as genetic algorithms, which work with a population of candidate solutions and move this population toward better feasible solutions by evaluating multiple performance measures iteratively \citep{udhayakumar2011stochastic,tsai2014genetic,chang2025maximizing}. However, there are generally no theoretical guarantees for such SO methods. Another line of research focuses on model-based methods, mainly referred to as Bayesian optimization (BO) \citep{pmlr-v130-eriksson21a}. They build probabilistic surrogate models for the objective and constraint functions, including constrained expected improvement (CEI) \citep{gardner2014bayesian}, predictive entropy search with constraints \citep{hernandez2016general}, and constrained knowledge gradient (CKG) \citep{ungredda2024bayesian}. Although BO methods can perform well in moderate-dimensional settings, they generally suffer from inefficiency in high-dimensional problems \citep{wang2023recent}. In addition, due to the complexity of the algorithm dynamics, such as CEI and CKG, it is generally difficult to establish non-asymptotic convergence guarantees for BO methods \citep{srinivas2010gaussian,ungredda2024bayesian}.

Outside the simulation literature, the min-max problem studied in this paper is also closely connected to zeroth-order min-max optimization in the nonconvex-concave settings \citep{liu2020min,NEURIPS2024_413885e7}.  The best-known sample complexity $O(d\epsilon^{-4})$ was established in \cite{NEURIPS2024_413885e7} to derive an $\epsilon$-stationary point for the stochastic nonconvex-concave problems. Our approach can also apply to such settings and improves the best-known sample complexity to $\tilde{O}(d\epsilon^{-3.5})$. 
Moreover, we point out that this research emphasizes the simulation setting, in which the algorithm has access only to simulation samples and where simulation techniques such as CRNs can play an important role in its finite-time performance.

The remainder of this study is organized as follows. We introduce the methodology in Section \ref{sec:Method} and establish the finite-iteration convergence guarantees in Section \ref{sec:convergence_analysis}. Then, we present the results of numerical experiments in Section \ref{sec:numerical}. Finally, Section \ref{sec:conclusion} concludes this paper.

\section{Methodology}\label{sec:Method}

This section presents our methodology for solving \eqref{eq:constrained_problem}. We first present the alternative min-max problem of \eqref{eq:constrained_problem} and provide the technical assumptions. We then introduce the gradient estimator and use it to design the MGS algorithm.

\paragraph{Notation} For a differentiable function $h(x,y):\mathbb{R}^{d_x}\times \mathbb{R}^{d_y}\to \mathbb{R}$, denote the partial gradients with respect to $x$ and $y$ as $\nabla_x h(x,y)$ and $\nabla_y h(x,y)$, respectively. Without further specification, $\|\cdot\|$ denotes the $\ell_2$-norm in Euclidean space. For some random variable $\varsigma$, $\mathbb{E}_\varsigma[\cdot]$ denotes the expectation with respect to $\varsigma$. For any convex set $\mathcal{Z}$, the operator $\mathcal{P}_\mathcal{Z}[\cdot]$ denotes the projection to the set $\mathcal{Z}$.

\subsection{Problem Formulation and Assumptions}
Since the constraints in problem \eqref{eq:constrained_problem} are non-analytical and stochastic, it is challenging to handle them using traditional projection-based SA methods. Instead, we derive its Lagrangian and consider the following min–max formulation:
\begin{align}\label{eq:minmax_problem}
\min_{x\in\mathcal{X}}\max_{y\in\mathcal{Y}}\; h(x,y)=\mathbb{E}\left[h(x,y;\xi)\right],
\end{align}
where $h(x,y)$ is the Lagrangian and $h(x,y;\xi)=h_0(x;\xi)+\sum_{j\in\mathcal{I}}y(j)h_j(x;\xi)$.
The vector $y\in \mathbb{R}^{d_y}$ is the Lagrangian multiplier (we denote $d_y=|\mathcal{I}|$ and $d=d_x+d_y$).
The natural dual cone is
$\mathcal{Y}_{+}:=\{y\in\mathbb{R}^{d_y}\mid y\ge 0\}$.
In our algorithm and analysis, we work with a bounded truncation $\mathcal{Y}:=\{y\in \mathcal{Y}_{+}\mid 0\le y\le \bar y\}$,
where the inequality is component-wise, and $\bar y> 0$ is a fixed upper bound.
% The vector $y\in\mathbb{R}^{d_y}$ is the Lagrangian multiplier (we have $d_y=|\mathcal{I}|$), and $\mathcal{Y}=\{y\in \mathbb{R}^{d_y}| y\geq 0\}$.  
Generally, the stationary points of $h(x,y)$ can provide high-quality solutions to \eqref{eq:constrained_problem} under mild conditions.
We will focus on solving  \eqref{eq:minmax_problem} using SO methods in the rest of this research.

We make the following technical assumptions for our analysis. 

\begin{assumption}\label{assump:bounded}
% For any $\mu\geq 0$, $\inf_{x\in\mathcal{X}}\sup_{y\in\mathcal{Y}} f(x,y)$ is lower bounded by some $F_\mu^*>-\infty$.
$\mathcal{X}$ and $\mathcal{Y}$ are convex and compact sets.
% Moreover, $h(x,y)$ is lower bounded, i.e., $h^*=\min_{(x,y) \in \mathcal{X}\times\mathcal{Y}} h(x,y) \geq -\infty$.
\end{assumption}
Assumption \ref{assump:bounded} imposes boundedness on $\mathcal{X}$ and $\mathcal{Y}$.
The boundedness of $\mathcal{X}$ is a natural assumption in practice. Note that the natural dual cone $\mathcal{Y}_{+}=\mathbb{R}^{d_y}_{+}$ is unbounded. For the dual update, however, we consider the bounded truncation $\mathcal{Y}=\{y\in\mathbb{R}^{d_y}\mid 0\le y\le \bar y\}$ for algorithmic stability and finite-time analysis. This is motivated by the fact that the optimal dual set is bounded when a strictly feasible solution exists for \eqref{eq:constrained_problem} \citep{nedic2009subgradient}. Consequently, we can choose $\mathcal{Y}$ large enough so that it contains at least one optimal dual variable and then solve \eqref{eq:minmax_problem} over $\mathcal{X}\times\mathcal{Y}$.
Although solving \eqref{eq:minmax_problem} is not exactly the same as solving \eqref{eq:constrained_problem}, we will show that any stationary point of \eqref{eq:minmax_problem} corresponds to a KKT point of \eqref{eq:constrained_problem} under additional conditions.

\begin{assumption}\label{assump:lipschitz}
Each component function $h(x,y;\xi)$ is differentiable and Lipschitz continuous, i.e., for any $x_1, x_2\in\mathbb{R}^{d_x}$ and $y_1,y_2 \in\mathbb{R}^{d_y}$, we have $| h(x_1,y_1;\xi)-h(x_2,y_2;\xi)| \leq \Lambda\|(x_1,y_1)-(x_2,y_2)\|$ for some $\Lambda>0$. 
\end{assumption}

\begin{assumption}\label{assump:smoothness}
Each component function $h(x,y;\xi)$ has $L_h$-Lipschitz continuous gradients, i.e., there exists a $L_h>0$ such that $\|\nabla_x h(x_1,y_1;\xi)-\nabla_x h(x_2,y_2;\xi)\|\leq L_h\|(x_1,y_1)-(x_2,y_2)\|$ and $\|\nabla_y h(x_1,y_1;\xi)-\nabla_y h(x_2,y_2;\xi)\|\leq L_h\|(x_1,y_1)-(x_2,y_2)\|$ for any $x_1,x_2\in\mathbb{R}^{d_x}$ and $y_1,y_2\in\mathbb{R}^{d_y}$.
% \begin{align*}
%     \|\nabla_x h(x_1,y_1;\xi)-\nabla_x h(x_2,y_1;\xi)\|&\leq L_h\|x_1-x_2\|,\\
%     \|\nabla_x h(x_1,y_1;\xi)-\nabla_x h(x_1,y_2;\xi)\|&\leq L_h\|y_1-y_2\|,\\
%     \|\nabla_y h(x_1,y_1;\xi)-\nabla_y h(x_1,y_2;\xi)\|&\leq L_h\|y_1-y_2\|,\\
%     \|\nabla_y h(x_1,y_1;\xi)-\nabla_y h(x_2,y_1;\xi)\|&\leq L_h\|x_1-x_2\|.
% \end{align*}
\end{assumption}
\begin{assumption}\label{assump:bounded_variance}
    There exists some $\Tilde{\sigma}\geq 0$ such that $\mathbb{E}_{\xi}[\|\nabla_x h(x,y;\xi)-\nabla_x h(x,y)\|^2]\leq \Tilde{\sigma}^2$ and $\mathbb{E}_{\xi}[\|\nabla_y h(x,y;\xi)-\nabla_y h(x,y)\|^2]\leq \Tilde{\sigma}^2$.
\end{assumption}

Assumptions \ref{assump:lipschitz} and \ref{assump:smoothness} impose Lipschitz continuity on both component functions $h(x,y;\xi), \forall\xi$ and their gradients, which are common for search-based SO problems \citep{fan2018surrogate,hu2025convergence}. Assumption \ref{assump:bounded_variance} serves to bound the variance of stochastic gradients. It is satisfied when we alternatively assume the variances of $h_j(x;\xi),\forall j\in \mathcal{I}$ and $\nabla_x h_j(x;\xi),\forall j\in \{0\}\cup \mathcal{I}$ to be bounded. Similar assumptions are widely adopted in gradient-based SO \citep{hu2022stochastic,xu2023gradient}.

% the stationary points are not equivalent when applying $\mathcal{Y}$ and $\mathcal{Y}_+$, the optimal dual set under Slater condition is bounded \citep{nedic2009subgradient}.
% The rationale comes from the boundedness of the optimal dual set under Slater condition \citep{nedic2009subgradient}. 

% Alternatively, we can assume Slater's condition holds, under which one can specify a $\bar y$ so that $\mathcal{Y}$ contains at least one maximizer of the Lagrangian dual function \citep{nedic2009subgradient}.
% The boundedness on $\mathcal{X}$ is natural in practice, but there is a gap for the boundedness of $\mathcal{Y}$ between the theoretical analysis and our problem where $\mathcal{Y}=\{y\in\mathbb{R}^{d_y}| y\geq 0\}$ is unbounded. It has been shown that the optimal dual set is bounded under the Slater condition (there exists some $x\in\mathcal{X}$ satisfying $h_j(x)< 0,\forall j\in\mathcal{I}$) \citep{nedic2009subgradient}. Therefore, we can construct such a bounded set to replace $\mathcal{Y}$ in our algorithm.
% we can construct a bounded set containing the optimal set of the dual variables to replace $\mathcal{Y}$ under the Slater condition (there exists some $x\in\mathcal{X}$ satisfying $h_j(x)< 0,\forall j\in\mathcal{I}$) \citep{nedic2009subgradient}.
% It is a common issue in the analysis of gradient descent ascent algorithms, although it is proven in \cite{nedic2009subgradient} that the optimal $y$ lies in a bounded set when the $h(x,y)$ is convex in $x$.

\subsection{Finite-Difference Gradient Estimator}
We propose to use the gradient-based method to solve problem \eqref{eq:constrained_problem} within the primal-dual framework. Due to the non-analytical property of $h_j(x), \forall j\in \{0\}\cup \mathcal{I}$, direct gradient calculation of $h(x,y)$ is intractable. Alternatively, we can estimate the gradient of $h(x,y)$ by the input-output information derived from simulations. Traditional coordinate-wise finite-difference gradient estimator requires $O(d)$ simulations to construct a gradient estimate, which is expensive for high-dimensional problems \citep{scheinberg2022finite}. Therefore, we apply a random gradient estimator based on Gaussian smoothing \citep{nesterov2017random,xu2023gradient,lam2025distributionally}. Denote $\zeta=\{\xi, z, e\}$, where the vectors $z\in\mathbb{R}^{d_x}$ and $e\in\mathbb{R}^{d_y}$ are sampled from the Gaussian distributions $\mathcal{N}(0,d_x^{-1}I_{d_x})$ and $\mathcal{N}(0,d_y^{-1}I_{d_y})$, respectively. Consider the following forward finite-difference gradient estimator:
\begin{align}
&\hat{\nabla}_x h(x,y;\zeta)=\frac{d_x(h(x+rz,y;\xi)-h(x,y;\xi))}{r}\cdot z,\label{eq:grad_x_single}\\
&\hat{\nabla}_y h(x,y;\zeta)=\frac{d_y(h(x,y+re;\xi)-h(x,y;\xi))}{r}\cdot e,\label{eq:grad_y_single}
\end{align}
where $\hat{\nabla}_x h(x,y;\zeta)$ and $\hat{\nabla}_y h(x,y;\zeta)$ are estimated partial gradients in $x$ and $y$, respectively. $r>0$ is a small scalar called the smoothing radius. Here, we apply CRNs within each estimation: the baseline and perturbed simulations are driven by the same realization of $\xi$ so that $ h(x+rz, y; \xi)$ and $ h(x, y; \xi)$ share randomness (and similarly for $h(x,y+re;\xi)$ and $h(x,y;\xi)$). This coupling turns the two-point difference into an increment along the same sample path and reduces the variance of the resulting gradient estimator.
% in the two simulation runs used to compute $\hat{\nabla}_x h(x, y; \zeta)$ and $\hat{\nabla}_y h(x, y; \zeta)$ for variance reduction, i.e., $h(x+rz,y;\xi)$ and $h(x,y;\xi)$ are based on the same $\xi$ (and so are $h(x,y+re;\xi)$ and $h(x,y;\xi)$). 
It is also worth noting that the partial gradient in $y(j)$ can be denoted by $ h_j(x;\xi)$ when $h(x,y)$ is the Lagrange function of \eqref{eq:constrained_problem}. However, we apply the more general form of \eqref{eq:grad_y_single} in our algorithm design and analysis, which does not essentially change the order of performance guarantees.

% For each gradient estimation, we sample a mini-batch of $q$ independent random vectors $\mathcal{B}=\{\xi_i\}_{i=1}^q$. Then, consider the classical two-point gradient estimator for the partial gradient w.r.t $x$ based on the finite difference of two simulation outputs:
% \begin{align}\label{eq:grad_x_estimator}
% \hat{\nabla}_x h(x,y;\mathcal{B})\!=\!\frac{1}{q}\!\sum_{i=1}^q\!\frac{h(x+rz_i,y;\xi_i)\!-\!h(x,y;\xi_i)}{r}\!\cdot\! z_i,
% \end{align}
% where $r>0$ is a small scalar called the smoothing radius. The vector $z_i\in\mathbb{R}^{d_x}$ is sampled from the Gaussian distribution $\mathcal{N}(0,I_{d_x})$. Note that two variance-reduction techniques are used in \eqref{eq:grad_x_estimator}. First, we use the common random number to carry on the two simulations, i.e., $\xi_i$ is generated using the same random seeds. Second, a mini-batch of $q$ stochastic gradient estimations is generated and averaged for each estimator.

% Similarly, the partial gradient w.r.t. $y$ can be formulated as
% \begin{align}\label{eq:grad_y_estimator}
% \hat{\nabla}_y h(x,y;\mathcal{B})\!=\!\frac{1}{q}\!\sum_{i=1}^q\!\frac{h(x,y\!+re_i;\xi_i)\!-\!h(x,y;\xi_i)}{r}\!\cdot\! e_i,
% \end{align}
% where $e_i\in\mathbb{R}^{d_y}$ is the vector sampled from $\mathcal{N}(0,I_{d_y})$. 

To further investigate the properties of gradient estimators \eqref{eq:grad_x_single} and \eqref{eq:grad_y_single}, define
$$h_1(x,y;\xi)=\mathbb{E}_z[h(x+rz,y;\xi)],\qquad h_2(x,y;\xi)=\mathbb{E}_e[h(x,y+re;\xi)]$$
as the expected gradient estimates given $\xi$.
In addition, denote $h_1(x,y)=\mathbb{E}_\xi[h_1(x,y;\xi)]$ and $h_2(x,y)=\mathbb{E}_\xi[h_2(x,y;\xi)]$. Under mild conditions, we obtain the following lemma, which establishes bias bounds for the gradient estimators \eqref{eq:grad_x_single} and \eqref{eq:grad_y_single} and will be frequently used in our analysis.
\begin{lemma}[\cite{nesterov2017random}]\label{lemma:esti_bias}
Under Assumption \ref{assump:smoothness}, we have (1) $h_1(x,y;\xi)$ and $h_2(x,y;\xi)$ are $L_h$-smooth for any $\xi$; (2) $\mathbb{E}_{z}[\hat{\nabla}_x h(x,y;\zeta)]=\nabla_x h_1(x,y;\xi)$ and $\mathbb{E}_{e}[\hat{\nabla}_y h(x,y;\zeta)]=\nabla_y h_2(x,y;\xi)$ for any $\xi$; (3) $\|\nabla_x h_1(x,y;\xi)-\nabla_x h(x,y;\xi)\|\leq L_hr$ and $\|\nabla_y h_2(x,y;\xi)-\nabla_y h(x,y;\xi)\|\leq L_hr$.
\end{lemma}
Lemma \ref{lemma:esti_bias} indicates that the gradient estimators \eqref{eq:grad_x_single} and \eqref{eq:grad_y_single} are biased estimators and the bias can be very small when a small smoothing radius is applied. Moreover, we provide the following lemma to control the variance of gradient estimators \eqref{eq:grad_x_single} and \eqref{eq:grad_y_single}.
\begin{lemma}\label{lemma:esti_variance}
Under Assumptions \ref{assump:smoothness} and \ref{assump:bounded_variance}, there exists some $\sigma>0$ satisfying $\mathbb{E}_{z,\xi}[\|\hat{\nabla}_x h(x,y;\zeta)-\nabla_x h_1(x,y)\|^2]\leq \sigma^2$ and $\mathbb{E}_{e,\xi}[\|\hat{\nabla}_y h(x,y;\zeta)-\nabla_y h_2(x,y)\|^2]\leq \sigma^2$.
\end{lemma}
\begin{proof}[Proof of Lemma \ref{lemma:esti_variance}]
For the first inequality, we decompose the variance by
\begin{align}
&\mathbb{E}_{z,\xi}[\|\hat{\nabla}_x h(x,y;\zeta)-\nabla_x h_1(x,y)\|^2]\nonumber\\
=\;&\mathbb{E}_{z,\xi}[\|\hat{\nabla}_x h(x,y;\zeta)- \mathbb{E}_z[\hat{\nabla}_x h(x,y;\zeta)]\nonumber+\mathbb{E}_z[\hat{\nabla}_x h(x,y;\zeta)]-\nabla_x h_1(x,y)\|^2]\nonumber\\
=\; &\mathbb{E}_{z,\xi}[\|\hat{\nabla}_x h(x,y;\zeta)- \nabla_x h_1(x,y;\xi)\|^2]+\mathbb{E}_{\xi}[\|\nabla_x h_1(x,y;\xi)-\nabla_x h_1(x,y)\|^2],\label{eq:proof_lemma_2}
\end{align}
where the last step follows from Lemma \ref{lemma:esti_bias} and the fact 
$$\mathbb{E}_{z,\xi}[\langle \hat{\nabla}_x h(x,y;\zeta)- \nabla_x h_1(x,y;\xi), \nabla_x h_1(x,y;\xi)-\nabla_x h_1(x,y) \rangle]=0.$$
The first term in the right-hand side of \eqref{eq:proof_lemma_2} can be bounded by $3d_x\Lambda^2+\frac{1}{4}(d_x+2)(d_x+4)L_h^2r^2$ (see Lemma 2.4 in \cite{berahas2022theoretical}). The second term is bounded by 
% Requires: \usepackage{amsmath}
\begin{align*}
    \mathbb{E}_{\xi}\bigl[\|\nabla_{x} h_{1}(x,y;\xi) - \nabla_{x} h_{1}(x,y)\|^{2}\bigr]
    = \; & \mathbb{E}_{\xi}\Bigl[\bigl\|\mathbb{E}_{z}\bigl[\nabla_{x} h(x+rz,y;\xi) - \nabla_{x} h(x+rz,y)\bigr]\bigr\|^{2}\Bigr]\\
    \le\; & \mathbb{E}_{z,\xi}\bigl[\|\nabla_{x} h(x+rz,y;\xi) - \nabla_{x} h(x+rz,y)\|^{2}\bigr]
    \le \tilde{\sigma}^{2},
\end{align*}
where we used Jensen's inequality and Assumption \ref{assump:bounded_variance}. Set $\sigma^2=3d\Lambda^2+\frac{1}{4}(d+2)(d+4)L_h^2r^2+\Tilde{\sigma}^2$. Then, we have $\mathbb{E}_{z,\xi}[\|\hat{\nabla}_x h(x,y;\zeta)-\nabla_x h_1(x,y)\|^2]\leq \sigma^2$. The second inequality in Lemma \ref{lemma:esti_variance} follows a similar proof and shares the same $\sigma^2$ bound.
\Halmos
\end{proof}
The results in Lemma \ref{lemma:esti_variance} show that \eqref{eq:grad_x_single} and \eqref{eq:grad_y_single} have bounded variances under standard assumptions. In our algorithm, variance plays a crucial role in determining the performance of our method.

\subsection{Variance Reduction of CRN Implementation}\label{subsec:crn_analysis}
The variance bound in Lemma \ref{lemma:esti_variance} reflects an important implementation choice, namely, whether the baseline and perturbed function evaluations in \eqref{eq:grad_x_single}-\eqref{eq:grad_y_single} use CRN or independent randomness.
% For the gradient estimator based on independent randomness, we can derive a variance bound analogous to that in Lemma \ref{lemma:esti_variance}.
% To distinguish the two variance constants,
% let $\sigma_{\mathrm{CRN}}^{2}$ and $\sigma_{\mathrm{IND}}^{2}$
% denote the variance bounds associated with the CRN-based implementation and the independent randomness, respectively.
The effect of CRN originates at the level of the underlying two-point simulation difference. For notational simplicity, let $A=h(x+rz,y;\xi)$, $B=h(x,y;\xi)$, and $B'=h(x,y;\xi')$, where $\xi$ and $\xi'$ are independent random vectors. Moreover, define $g_{\mathrm{CRN}}=\hat{\nabla}_x h(x,y;\zeta)=\frac{d_x(A-B)}{r}\cdot z$ and $g_{\mathrm{IND}}=\frac{d_x(A-B')}{r}\cdot z$. We focus only on the gradient estimator for $\nabla_x h(x,y)$, since the argument for the estimator of $\nabla_y h(x,y)$ is entirely analogous. 

The following proposition establishes the upper and lower bounds on $\operatorname{Var}(g_{\mathrm{IND}})$.

\begin{proposition}\label{prop:var_bounds}
Suppose Assumption \ref{assump:lipschitz}-\ref{assump:bounded_variance} hold and $\underline{\sigma}^2\leq \operatorname{Var}(A),\operatorname{Var}(B)\leq \overline{\sigma}^2$ for some $\underline{\sigma}, \overline{\sigma}>0$ and $(x,y)\in\mathcal{X}\times\mathcal{Y}$. Then, we have $\frac{\underline{\sigma}^2d_x^2}{r^2}\leq \operatorname{Var}(g_{\mathrm{IND}})\leq \sigma^2+\frac{2\overline{\sigma}^2d_x^2}{r^2}$.
% $\operatorname{Var}(g_{\mathrm{CRN}})\leq (d_x+2)d_x\Lambda^2$
\end{proposition}
\begin{proof}[Proof of Proposition \ref{prop:var_bounds}]
% Conditioning on the perturbation direction $z$, we have
% $$
% \operatorname{Var}(g_{\mathrm{CRN}}\mid z) =\operatorname{Var}(\frac{d_x(A-B)}{r}\cdot z\mid z)=\frac{d_x^2\|z\|^2}{r^2}\operatorname{Var}(A-B\mid z)
% \leq \frac{d_x^2\|z\|^2}{r^2}\mathbb{E}[(A-B)^2\mid z].
% $$
% Using the inequality $|A-B|\leq \Lambda r\|z\|$ derived from Assumption \ref{assump:lipschitz}, we can get
% $$\operatorname{Var}(g_{\mathrm{CRN}})\leq \mathbb{E}[d_x^2\Lambda^2\|z\|^4]=d_x(d_x+2)\Lambda^2,$$
% where we also used the fact $\mathbb{E}[\|z\|^4]=1+2/d_x$.

Conditioning on the perturbation direction for the variance of $g_{\mathrm{IND}}$, we can obtain
$$\operatorname{Var}(g_{\mathrm{IND}}\mid z)=\operatorname{Var}\big(\frac{d_x(A-B')}{r}\cdot z \mid z\big)=\frac{d_x^2\|z\|^2}{r^2}\operatorname{Var}(A-B'\mid z)=\frac{d_x^2\|z\|^2}{r^2}\big(\operatorname{Var}(A\mid z)+\operatorname{Var}(B'\mid z)\big),$$
Using the condition that $\operatorname{Var}(B')\geq \underline{\sigma}^2$, we can derive
$$\operatorname{Var}(g_{\mathrm{IND}})\geq\mathbb{E}\big[\frac{d_x^2\|z\|^2\underline{\sigma}^2}{r^2}\big]=\frac{d_x^2\underline{\sigma}^2}{r^2},$$
where we used the fact $\mathbb{E}[\|z\|^2]=1$. As for the upper bound of $\operatorname{Var}(g_{\mathrm{IND}})$, we can decompose the variance by
\begin{align*}
&\mathbb{E}[\|g_{\mathrm{IND}}-\nabla_x h_1(x,y)\|^2]\\
=\; & \mathbb{E}[\|g_{\mathrm{IND}}-\mathbb{E}_{\xi,\xi'}[g_{\mathrm{IND}}]+\mathbb{E}_{\xi,\xi'}[g_{\mathrm{IND}}]-\nabla_x h_1(x,y)\|^2]\\
=\;& \mathbb{E}[\|g_{\mathrm{IND}}-\mathbb{E}_{\xi,\xi'}[g_{\mathrm{IND}}]\|^2]+\mathbb{E}[\|\hat{\nabla}_x h(x,y)-\nabla_x h_1(x,y)\|^2],
\end{align*}
where we used the facts $$\mathbb{E}[\langle g_{\mathrm{IND}}-\mathbb{E}_{\xi,\xi'}[g_{\mathrm{IND}}] , \hat{\nabla}_x h(x,y)-\nabla_x h_1(x,y) \rangle]=0,$$ 
and $\mathbb{E}_{\xi, \xi'}[g_{\mathrm{IND}}]=\hat{\nabla}_x h(x,y)$. The second term on the right-hand side can similarly be bounded by $3d_x\Lambda^2+\frac{1}{4}(d_x+2)(d_x+4)L_h^2r^2$. The first term can be further processed by
$$
\mathbb{E}[\|g_{\mathrm{IND}}-\mathbb{E}_{\xi,\xi'}[g_{\mathrm{IND}}]\|^2]=\mathbb{E}\left[\left\|\frac{d_x z}{r}\big(A-B'-h(x+rz,y)+h(x,y)\big) \right\|^2\right]\leq \frac{2d_x^2\overline{\sigma}^2}{r^2}.$$
Therefore, we can get the upper bound of $\sigma^2+\frac{2d_x^2\overline{\sigma}^2}{r^2}$ for $\operatorname{Var}(g_{\mathrm{IND}})$. \Halmos
\end{proof}
Proposition \ref{prop:var_bounds} shows that the lower bound of $\operatorname{Var}(g_{\mathrm{IND}})$ is proportional to $r^{-2}$. Since $r$ controls the perturbation magnitude and is often set very small to reduce the estimation bias, $\operatorname{Var}(g_{\mathrm{IND}})$ can be very large under independent randomness. However, the upper bound of $\operatorname{Var}(g_{\mathrm{CRN}})$ is $\sigma^2$ as in Lemma \ref{lemma:esti_variance}, which can significantly reduce the variance when $r$ is small.
The following corollary further characterizes the variance reduction effect from the use of CRNs.

\begin{corollary}\label{coro:crn_VR}
Suppose that the conditions in Proposition \ref{prop:var_bounds} hold. Set $r\leq \frac{\underline{\sigma}d_x}{\sigma }$. Then, we have $\operatorname{Var}(g_{\mathrm{IND}})\geq \operatorname{Var}(g_{\mathrm{CRN}})$.    
\end{corollary}

Corollary \ref{coro:crn_VR} indicates that using a sufficiently small smoothing radius in the CRN setting results in a variance not higher than that of independent randomness. Note that only the case of $\underline{\sigma}>0$ is considered here, since it is obvious
that $\operatorname{Var}(g_{\mathrm{IND}})= \operatorname{Var}(g_{\mathrm{CRN}})$ if $\underline{\sigma}=0$ for some $(x,y)\in\mathcal{X}\times\mathcal{Y}$, regardless of the value of $r$.
Let $\sigma_{\mathrm{CRN}}^2$ denote the upper variance bound of gradient estimation for CRN implementation, as established in Lemma \ref{lemma:esti_variance}, and let $\sigma_{\mathrm{IND}}^2$ denote the corresponding upper bound for independent randomness. Proposition \ref{prop:var_bounds} and Corollary \ref{coro:crn_VR} suggest that we can select a better variance bound for $\sigma_{\mathrm{CRN}}^2$ as in Lemma \ref{lemma:esti_variance} for the CRN implementation that is no larger than $\sigma_{\mathrm{IND}}^2$.

\subsection{Min-Max Gradient Search Algorithm} \label{subsec:mgs_algo}
\begin{algorithm}[htbp]
	\caption{Min-max gradient search (MGS)}
	\label{algorithm:MGS} 
	\begin{algorithmic}[1]
		\State \textbf{Input:} Initial $x_1,y_1$, maximum steps $T$, and the parameters $\{\eta_t\}_{t=1}^{T}, \{\alpha_t, \beta_t\}_{t=2}^{T+1}$ and $\{\mu, \gamma, \lambda\}$.
        \State Set $f(x,y):= h(x,y)-\frac{\mu}{2}\|y\|^2$. Sample an initial set $\mathcal{S}_1=\{\zeta_1^i\}_{i=1}^q$ and set $v_1=\hat{\nabla}_x f(x_1,y_1;\mathcal{S}_1),w_1=\hat{\nabla}_y f(x_1,y_1;\mathcal{S}_1)$.
        \For{$t \gets 1$ \textbf{to} $T$}
        \State Update $x_t$ by: $x_{t+1}=x_t+\eta_t(\Tilde{x}_{t+1}-x_t)$, where $\Tilde{x}_{t+1}=\mathcal{P}_\mathcal{X}[x_t-\gamma v_t]$.
        \State Update $y_t$ by: $y_{t+1}=y_t+\eta_t(\Tilde{y}_{t+1}-y_t)$, where $\Tilde{y}_{t+1}=\mathcal{P}_{\mathcal{Y}}[y_t+\lambda w_t]$.
        \State Draw a set of i.i.d. samples $\mathcal{S}_{t+1}=\{\zeta^i_{t+1}\}_{i=1}^q$, and calculate
        $$v_{t+1}=\hat{\nabla}_x f(x_{t+1},y_{t+1};\mathcal{S}_{t+1})+(1-\alpha_{t+1})\big(v_t-\hat{\nabla}_x f(x_t,y_t;\mathcal{S}_{t+1})\big),$$
        $$w_{t+1}= \hat{\nabla}_y f(x_{t+1},y_{t+1};\mathcal{S}_{t+1})+(1-\beta_{t+1})\big(w_t-\hat{\nabla}_y f(x_t,y_t;\mathcal{S}_{t+1})\big).$$
	\EndFor
        \State \textbf{Output:} $\{(x_t,y_t)\}_{t=1}^{T+1}$
	\end{algorithmic}
\end{algorithm}
In contrast to gradient descent approaches designed for projection-friendly constraints \citep{peng2016gradient,cakmak2021solving,hu2022stochastic}, the proposed min-max gradient search (MGS) algorithm (see Algorithm \ref{algorithm:MGS}) is built on the GDA framework, in which we apply gradient descent to the primal variables and gradient ascent to the dual variables.  
The MGS algorithm is built on two main ideas. First, it is well established that standard GDA enjoys improved convergence guarantees when $h(x,y)$ is strongly concave in $y$ \citep{lin2020gradient}. Inspired by the alternative gradient projection (AGP) method \citep{xu2023unified}, we introduce the auxiliary function
$$f(x,y):=h(x,y)-\frac{\mu}{2}\|y\|^2,$$
where $\mu>0$. In situations where $h(x,y)$ is only concave in $y$, such as $h(x,y)=x^{T}y$, applying vanilla GDA to $h(x,y)$ may fail to converge and instead enter limit cycles. 
The regularized function $f(x,y)$ is $\mu$-strongly concave in $y$, and the strong-concavity regularization term acts as a damping factor in the dual ascent dynamics to mitigate oscillations and promote convergence. Generally, $\mu$ is chosen small to stabilize the dual ascent dynamics while limiting the distortion of the original Lagrangian. We then apply GDA to $f(x,y)$ rather than to $h(x,y)$ (see Lines 4 and 5). 
% is taken appropriately to introduce strong concavity to $f(x,\cdot)$ in $y$ while limiting the distortion of $h(x,y)$. 

The second idea is to apply the batched update to further reduce variance and use momentum-based updates to accelerate convergence. The batched update involves taking an average over multiple replications of gradient computations with varying CRNs to further reduce the variance. Specifically, for each gradient estimation, we sample $q> 0$ independent sets of random vectors $\mathcal{S}=\{\zeta^i\}_{i=1}^q=\left\{\{\xi^i,z^i,e^i\}_{i=1}^q\right \}$, and calculate \eqref{eq:grad_x_single} and \eqref{eq:grad_y_single} for $q$ replications. Then, we get the averaged gradient estimation:
$$\hat{\nabla}_x h(x,y;\mathcal{S})=\frac{1}{q}\sum_{i=1}^q\hat{\nabla}_x h(x,y;\zeta^i), \quad \hat{\nabla}_y h(x,y;\mathcal{S})=\frac{1}{q}\sum_{i=1}^q\hat{\nabla}_y h(x,y;\zeta^i).$$
In MGS, we can easily derive $\hat{\nabla}_x f(x,y;\mathcal{S})=\hat{\nabla}_x h(x,y;\mathcal{S})$ and $\hat{\nabla}_y f(x,y;\mathcal{S})=\hat{\nabla}_y h(x,y;\mathcal{S})-\mu y$.
Momentum-based updates (Line 6) combine current and historical gradient information to smooth out noisy updates and advance along directions of consistent improvement, which can accelerate convergence \citep{cutkosky2019momentum}.

\section{Convergence Analysis}\label{sec:convergence_analysis}
This section presents the non-asymptotic analysis of the proposed MGS algorithm for solving problem \eqref{eq:minmax_problem} and some discussion about our results. 
\subsection{Non-asymptotic Convergence of MGS}
Unlike most of the gradient-based optimization literature that assumes (strong) convexity of the objective function \citep{xu2023gradient,hu2025convergence}, we focus on the nonconvex setting, as is often the case in SO. Stationarity is commonly used to characterize the performance of solutions for nonconvex problems, because it is a necessary condition for (local) optimality \citep{zhou2017gradient}. For constrained problems, a point is stationary if its negative gradient is orthogonal to feasible directions. A point $(x,y)$ is stationary for problem \eqref{eq:minmax_problem} if and only if $x=\mathcal{P}_\mathcal{X}[x-\gamma \nabla_x h(x,y)]$ and $y=\mathcal{P}_\mathcal{Y}[y+\lambda \nabla_y h(x,y)]$. Therefore, the following projected gradient is widely applied to measure the stationarity: $$\mathfrak{g}(x,y)=
\begin{pmatrix}
\mathfrak{g}_x(x,y)\\
\mathfrak{g}_y(x,y)
\end{pmatrix}
=
\begin{pmatrix}
    \frac{1}{\gamma}\left(x-\mathcal{P}_\mathcal{X}[x-\gamma \nabla_x h(x,y)]\right) \\
    \frac{1}{\lambda}\left(y-\mathcal{P}_\mathcal{Y}[y+\lambda \nabla_y h(x,y)]\right)
\end{pmatrix}$$ 
for any $(x,y)\in\mathcal{X}\times\mathcal{Y}$ \citep{xu2023unified,huang2022accelerated}. Any point $(x,y)$ satisfying $\|\mathfrak{g}(x,y)\|=0$ is a first-order stationary point of $h(x,y)$.
Then, we define the $\epsilon$-stationary point as follows.
\begin{definition}
We say a point $(x,y)\in\mathcal{X}\times\mathcal{Y}$ is $\epsilon$-stationary if $\|\mathfrak{g}(x,y)\|^2\leq \epsilon$.
\end{definition} 
% \begin{align*}
% &\mathfrak{g}(x_t,y_t)=
% \begin{pmatrix}
%     \mathfrak{g}_x(x_t,y_t) \\
%     \mathfrak{g}_y(x_t,y_t)
% \end{pmatrix}\\
% &=
% \begin{pmatrix}
%     \frac{1}{\gamma}\left(x_t-\mathcal{P}_\mathcal{X}[x_t-\gamma \nabla_x h(x_t,y_t)]\right) \\
%     \frac{1}{\lambda}\left(y_t-\mathcal{P}_\mathcal{Y}[y_t+\lambda \nabla_y h(x_t,y_t)]\right)
% \end{pmatrix}.
% \end{align*}
% This stationarity measure is widely applied for nonconvex optimization \citep{xu2023unified,huang2022accelerated}. 
% Define $F(x)=\max_{y\in\mathcal{Y}}f(x,y)$ and $y^*(x)=\arg\max_{y\in\mathcal{Y} }f(x,y)$. 
Similarly, we denote 
$f_{1}(x,y)=\mathbb{E}_{z}[f(x+rz,y)]$, and $f_{2}(x,y)=\mathbb{E}_{e}[f(x,y+re)]$ as the expected gradient estimates of $f(x,y)$. Furthermore, we define 
$$F_r(x)=\max_{y\in\mathcal{Y}}f_1(x,y),\quad y_r^*(x)=\arg\max_{y\in\mathcal{Y} }f_1(x,y), \quad F_r^*=\min_{x\in\mathcal{X}}F_r(x).$$ 
It can be easily shown that $\hat{\nabla}_x f(x,y;\zeta)$ (and $\hat{\nabla}_y f(x,y;\zeta)$) is also an unbiased estimator of $\nabla_x f_{1}(x,y)$ (and $\nabla_y f_{2}(x,y)$) with a finite variance bound $\sigma^2$. Define the constants $L=L_h+\mu$, $\kappa=L/\mu$, and $L_F=L+\kappa L$. Let $\Bar{y}\in\mathcal{Y}$ be the upper bound of $y$, and $D_y=\|\Bar{y}\|$. We provide our main results in the following theorem.

\begin{theorem}[Convergence Rate of MGS]\label{thm:Opt}
Suppose that Assumptions \ref{assump:bounded}-\ref{assump:bounded_variance} hold and the sequence $\{(x_t,y_t)\}_{t=1}^{T+1}$ is generated by Algorithm \ref{algorithm:MGS}. Set $\eta_t=(m+t)^{-\frac{1}{3}}, r\leq\frac{1}{L\sqrt{d}}(m+T)^{-\frac{2}{3}}$, and $\mu=(m+T)^{-\frac{2}{15-3e}}$ for some $e\in [0,1]$, where $m>0$ is set to satisfy $\mu\leq \min\big (\frac{2}{3},\frac{2L}{3}\big)$ and $\eta_0\leq \frac{1}{c}$ for some $c\geq \max \left( 2, 192L^2 \right)$. Let $\alpha_{t+1}=\beta_{t+1}=c\eta_t^2$, and $q=\lceil\frac{3(d+4)}{5L^2\mu^e}\rceil$. Let $ \gamma=k_1 \mu^{3-e}, \lambda=k_2 \mu^{1-e}$ where $k_1, k_2$ are positive constants satisfying $k_1\leq \min \left( \frac{1}{2L_F\eta_0\mu^{3-e}}, \frac{1}{4\sqrt{10}L^2\mu^{2-e}}, \frac{ k_2}{8\sqrt{5}L^2 }\right)$, and $\frac{\mu^e}{4L^2}\leq k_2\leq \min\left(\frac{1}{6L\mu^{1-e}},\frac{5q\mu^e}{12(d+4)}\right)$. Then, we have 
$$
\min_{t\leq T} \mathbb{E}\left[\left\|\mathfrak{g}(x_t,y_t)\right\|^2\right]\leq 
\frac{R_1}{T} (m+T)^{\frac{11-3e}{15-3e}}
+ \frac{R_2\ln(m+T)}{T}(m+T)^{\frac{9-e}{15-3e}}+R_3(m+T)^{-\frac{4}{15-3e}},
$$
where $R_1=\frac{16}{k_1}\big(F_r(x_1)+\frac{8L^2k_1\mu}{k_2} D_y^2+ \frac{2k_1\mu^{1-e}\sigma^2}{\eta_{0}q}
-F_r^*\big)$, $R_2=\frac{3200L^2}{d}+\frac{64c^2\sigma^2}{q}$, and $R_3=4D_y^2+\frac{7}{d}+\frac{720d}{q}$.
% \begin{align*}
% \min_{t\leq T} \mathbb{E}\left[\left\|\mathfrak{g}(x_t,y_t)\right\|^2\right]\leq O\left(T^{-\frac{4}{15-3e}}\right).
% \end{align*}
\end{theorem}
\begin{proof}[Proof of Theorem \ref{thm:Opt}]
For notational simplicity, we further define the following notation:
$$\Delta_x^t=\Tilde{x}_{t+1}-x_t,\quad \Delta_y^t=\Tilde{y}_{t+1}-y_t,\quad \Delta_*^t=y_r^*(x_t)-y_t,$$ 
$$\Delta_F^t=F_r(x_{t+1})-F_r(x_t),\quad \Delta_v^t=\nabla_x f_1(x_t,y_t)-v_t,\quad \Delta_w^t=\nabla_y f_2(x_t,y_t)-w_t.$$

The proof is roughly divided into \textit{three} steps. First, we try to derive a bound for $\|\mathfrak{g}(x_t,y_t)\|^2$ consisting of $\Delta_x^t, \Delta_y^t, \Delta_v^t,$ and $\Delta_w^t$. Then, we design a potential function $\phi (x_t)$ and establish a bound on the one-step drift of $\phi(x_t)$  using the four terms. Finally, we take the telescoping sum of the above bounds to derive the final bound for $\|\mathfrak{g}(x_t,y_t)\|^2$. We will elaborate on the three steps in detail. Note that the proofs of the following lemmas are provided in Appendix \ref{app:proofs_thm_opt}.

\paragraph{Step 1: Construct a bound on $\|\mathfrak{g}(x_t,y_t)\|^2$.} We provide this bound in the following lemma.
\begin{lemma}\label{lemma:measure_bound}
$\left\|\mathfrak{g}(x_t,y_t)\right\|^2\leq \frac{3}{\gamma^2}\|\Delta_x^t\|^2+3\|\Delta_v^t\|^2+\frac{4}{\lambda^2}\|\Delta_y^t\|^2+4\|\Delta_w^t\|^2+7L^2r^2+4\mu^2D_y^2.$
\end{lemma}

\paragraph{Step 2: Design a potential function and bound its one-step drift.} Define the potential function as 
$$\phi (x_t)=\mathbb{E}\left[F_r(x_t)+\frac{8L^2k_1\mu }{k_2}\|\Delta_*^t\|^2+\frac{\gamma}{\mu^2\eta_{t-1}}\left( \|\Delta_v^t\|^2+\|\Delta_w^t\|^2\right)\right],$$ 
for any $t\geq 1$. For simplicity, denote $\phi_t=\phi(x_t)$. Before establishing the bound on the one-step drift of $\phi_t$, we provide the following three lemmas to bound the one-step drift of $F_r(x_t), \|\Delta_*^t\|^2, \frac{1}{\eta_{t-1}}\|\Delta_v^t\|^2$, and $\frac{1}{\eta_{t-1}}\|\Delta_w^t\|^2$, respectively.
\begin{lemma}\label{lemma:F_diff_bound}
For any $t\geq 1$, $\Delta_F^t\leq 2 L^2\gamma \eta_t \|\Delta_*^t\|^2
+2\gamma\eta_t\|\Delta_v^t\|^2-\frac{\eta_t}{2\gamma}\|\Delta_x^t\|^2$.
% \begin{align*}
% \Delta_F^t\leq & 2 L^2\gamma \eta_t \|y_r^*(x_t)-y_t\|^2\\
% &+2\gamma\eta_t\|\Delta_v^t\|^2-\frac{\eta_t}{2\gamma}\|\Delta_x^t\|^2.
% \end{align*}
\end{lemma}
% Then, the following lemma provides the upper bound of one-step drift of $\Delta_*^t$.
\begin{lemma}\label{lemma:y_diff_bound}
$\|\Delta_*^{t+1}\|^2-\|\Delta_*^t\|^2
\leq  -\frac{\eta_t\mu \lambda}{4}\|\Delta_*^t\|^2\!-\!\frac{3\eta_t}{4}\|\Delta_y^t\|^2\!+\!\frac{12\eta_t\lambda}{\mu }\|\Delta_w^t\|^2
 +\frac{5\kappa^2 \eta_t}{\mu\lambda}\|\Delta_x^t\|^2+\frac{24L^2r^2\eta_t\lambda}{\mu}.$
% \begin{align*}
% &\|\Delta_*^{t+1}\|^2-\|\Delta_*^t\|^2\\
% \leq & -\frac{\eta_t\mu \lambda}{4}\|\Delta_*^t\|^2\!-\!\frac{3\eta_t}{4}\|\Delta_y^t\|^2\!+\!\frac{12\eta_t\lambda}{\mu }\|\Delta_w^t\|^2\\
% & +\frac{5\kappa^2 \eta_t}{\mu\lambda}\|\Delta_x^t\|^2+\frac{3L^2d_yr^2\eta_t\lambda}{\mu}.
% \end{align*}
\end{lemma}

\begin{lemma}\label{lemma:v_error_bound}
% Define the filtration $\mathcal{F}_t=\sigma((x_0,y_0),\mathcal{S}_0,(x_1,y_1),\cdots,\mathcal{S}_{t-1},(x_{t},y_{t}))$. Then, 
For any $t\geq 1$, we have
\begin{align*}
&\mathbb{E}\big[\frac{1}{\eta_t}\|\Delta_v^{t+1}\|^2-\frac{1}{\eta_{t-1}}\|\Delta_v^t\|^2\big]\\
\leq\;&  -\frac{\alpha_{t+1}}{\eta_t}\mathbb{E}\left[ \|\Delta_v^t\|^2 \right]+\frac{12L^2(d_x+2)\eta_t}{q}\mathbb{E}\left[ \|\Delta_x^t\|^2+\| \Delta_y^t \|^2 \right]
 +\frac{45L^2d_x^2 r^2}{q\eta_t}+\frac{2\alpha_{t+1}^2\sigma^2}{q\eta_t},\\
& \mathbb{E}\big[\frac{1}{\eta_t}\|\Delta_w^{t+1}\|^2-\frac{1}{\eta_{t-1}}\|\Delta_w^t\|^2\big]\\
\leq\; & -\frac{\beta_{t+1}}{\eta_t}\mathbb{E}\left[ \|\Delta_w^t\|^2 \right]+\frac{12L^2(d_y+2)\eta_t}{q}\mathbb{E}\left[ \|\Delta_x^t\|^2+\| \Delta_y^t \|^2 \right]
 +\frac{45L^2d_y^2 r^2}{q\eta_t}+\frac{2\beta_{t+1}^2\sigma^2}{q\eta_t}.
\end{align*}
% $$\mathbb{E}\left[\frac{1}{\eta_t}\|\Delta_v^{t+1}\|^2-\frac{1}{\eta_{t-1}}\|\Delta_v^t\|^2\right]
% \leq  -\frac{\alpha_{t+1}}{\eta_t}\mathbb{E}\left[ \|\Delta_v^t\|^2 \right]+\frac{12L^2(d_x+2)\eta_t}{q}\mathbb{E}\left[ \|\Delta_x^t\|^2+\| \Delta_y^t \|^2 \right]
%  +\frac{45L^2d_x^2 r^2}{q\eta_t}+\frac{2\alpha_{t+1}^2\sigma^2}{q\eta_t},$$
%  $$\mathbb{E}\left[\frac{1}{\eta_t}\|\Delta_w^{t+1}\|^2-\frac{1}{\eta_{t-1}}\|\Delta_w^t\|^2\right]
% \leq -\frac{\beta_{t+1}}{\eta_t}\mathbb{E}\left[ \|\Delta_w^t\|^2 \right]+\frac{12L^2(d_y+2)\eta_t}{q}\mathbb{E}\left[ \|\Delta_x^t\|^2+\| \Delta_y^t \|^2 \right]
%  +\frac{45L^2d_y^2 r^2}{q\eta_t}+\frac{2\beta_{t+1}^2\sigma^2}{q\eta_t}.$$
% \begin{align*}
% &\mathbb{E}\left[\frac{1}{\eta_t}\|\Delta_v^{t+1}\|^2-\frac{1}{\eta_{t-1}}\|\Delta_v^t\|^2\right]\\
% \leq & \!-\!\frac{\alpha_{t+1}}{\eta_t}\mathbb{E}\left[\! \|\Delta_v^t\|^2 \!\right]\!+\!\frac{6L^2\!(d_x\!+\!2)\eta_t}{q}\mathbb{E}\left[\! \|\Delta_x^t\|^2\!+\!\| \Delta_y^t \|^2 \!\right]\\
% & +\frac{3L^2d_x^2r^2}{q\eta_t}+\frac{2\alpha_{t+1}^2\sigma^2}{q\eta_t}.
% \end{align*}
% \begin{align*}
% &\mathbb{E}\left[\left.\frac{1}{\eta_t}\|\Delta_w^{t+1}\|^2-\frac{1}{\eta_{t-1}}\|\Delta_w^t\|^2\right| \mathcal{F}_{t}\right]\\
% \leq & \!-\!\frac{\beta_{t+1}}{\eta_t}\mathbb{E}\left[\! \|\Delta_w^t\|^2 \!\right]\!+\!\frac{6L^2\!(d_y\!+\!2)\eta_t}{q}\mathbb{E}\left[\! \|\Delta_x^t\|^2\!\!+\!\| \Delta_y^t \|^2 \!\right]\\
% & +\frac{3L^2d_y^2 r^2}{q\eta_t}+\frac{2\beta_{t+1}^2\sigma^2}{q\eta_t}.
% \end{align*}
\end{lemma}

Then, combining the results in Lemmas \ref{lemma:F_diff_bound}, \ref{lemma:y_diff_bound}, and \ref{lemma:v_error_bound}, we can derive a bound for $\phi_{t+1}-\phi_t$:
\begin{align}
\phi_{t+1}-\phi_t
% \leq &\mathbb{E}\left[\Delta_F^t +\frac{8L^2k_1\mu}{k_2}\left(\|\Delta_*^{t+1}\|^2-\|\Delta_*^t\|^2\right)\right]\\
% & + \frac{\gamma}{\mu^2}\mathbb{E}\left[\frac{1}{\eta_t}\|\Delta_v^{t+1}\|^2-\frac{1}{\eta_{t-1}}\|\Delta_v^t\|^2\right]\\
% & +\frac{\gamma}{\mu^2}\mathbb{E}\left[\frac{1}{\eta_t}\|\Delta_w^{t+1}\|^2- \frac{1}{\eta_{t-1}}\|\Delta_w^t\|^2\right]\\
\leq\;&  \mathbb{E}\left[\left(2\gamma\eta_t-\frac{\gamma\alpha_{t+1}}{\mu^2\eta_t}\right)\|\Delta_v^t\|^2 \right.+\left(\frac{40\kappa^2L^2 \eta_t \gamma}{\lambda^2 \mu^2}+\frac{12L^2\gamma(d+4)\eta_t}{ q \mu^2}-\frac{\eta_t}{2\gamma}\right)\|\Delta_x^t\|^2\nonumber\\
&+\left( \frac{12 L^2\gamma (d+4)\eta_t}{q \mu^2}-\frac{6L^2k_1\mu \eta_t }{k_2}\right)\|\Delta_y^t\|^2 + \left(\frac{96L^2\gamma\eta_t}{\mu^2}-\frac{\gamma\beta_{t+1}}{\mu^2\eta_t}\right)\|\Delta_w^t\|^2 \nonumber\\
&\left.+\frac{192L^4 r^2\gamma\eta_t}{\mu^2}+\frac{45L^2\!d^2r^2\gamma}{ q\eta_t\mu^2}+\frac{2(\alpha_{t+1}^2+\beta_{t+1}^2)\sigma^2\gamma}{q\eta_t\mu^2}\right].\label{eq:drift_bound1}
\end{align}
By $k_2 \leq \frac{5q\mu^e}{12(d+4)}$, we have $\frac{12 L^2\gamma(d+4)\eta_t}{q \mu^2}\leq \frac{5L^2k_1\mu\eta_t}{k_2}$. By $c\geq 192L^2$, we have $\frac{96L^2\gamma\eta_t}{\mu^2}\leq \frac{\gamma\beta_{t+1}}{2\mu^2\eta_t}$. Furthermore, we have $2\gamma\eta_t\leq \frac{\gamma\alpha_{t+1}}{2\mu^2\eta_t}$ due to $c\geq 2$ and $\mu\leq \frac{2}{3}$. Using the upper bound of $k_1$, we can get $\frac{40\kappa^2L^2 \eta_t \gamma}{\lambda^2 \mu^2}\leq \frac{\eta_t}{8\gamma}$ and $\frac{12L^2\gamma(d+4)\eta_t}{ q \mu^2}\leq \frac{\eta_t}{8\gamma}$, which lead to 
$$\frac{40\kappa^2L^2 \eta_t \gamma}{\lambda^2 \mu^2}+\frac{12L^2\gamma(d+4)\eta_t}{ q \mu^2}\leq \frac{\eta_t}{4\gamma}.$$
Then, we can rearrange the inequality \eqref{eq:drift_bound1} as
\begin{align*}
&\mathbb{E}\left[2\gamma\eta_t\|\Delta_v^t\|^2+\frac{\eta_t}{4\gamma}\|\Delta_x^t\|^2+\frac{\gamma\beta_{t+1}}{2\mu^2\eta_t} \|\Delta_w^t\|^2+\frac{L^2\gamma \eta_t}{\mu \lambda}\|\Delta_y^t\|^2\right]\\
\leq\;& \phi_t-\phi_{t+1}+\frac{192L^4 r^2\gamma\eta_t}{\mu^2}+\frac{45L^2d^2 r^2\gamma}{ q\eta_t\mu^2}+\frac{4c^2\eta_t^3\sigma^2\gamma}{q\mu^2}
\end{align*}
% \begin{align*}
% \phi_{t+1}-\phi_t \leq & \mathbb{E}\left[-2\gamma\eta_t\|\Delta_v^t\|^2-\frac{\eta_t}{4\gamma}\|\Delta_x^t\|^2-\frac{\gamma\beta_{t+1}}{2\mu^2\eta_t} \|\Delta_w^t\|^2\right.\\
% & -\frac{
% L^2\gamma \eta_t}{\mu \lambda}\|\Delta_y^t\|^2 \left.+\frac{96L^4d_y r^2\gamma\eta_t}{\mu^2}+\frac{45L^2d r^2\gamma}{ q\eta_t\mu^2}+\frac{4c^2\eta_t^3\sigma^2\gamma}{q\mu^2}\!\right].
% \end{align*}

Since $\frac{\mu^e}{4L^2}\leq k_2\leq \min\left(\frac{1}{6L\mu^{1-e}},\frac{5q\mu^e}{12(d+4)}\right)$, $c\geq 2$, and $\mu\leq \frac{2}{3}$, we have $\frac{4L^2}{\mu\lambda}\geq \frac{1}{\lambda^2}$ and $\frac{2\beta_{t+1}}{\mu^2\eta_t^2}\geq 1$. Then, combined with $\eta_t\geq \eta_T, \forall t\leq T$, the above inequality leads to
\begin{align}\label{eq:sum1}
&\mathbb{E}\big[\|\Delta_v^t\|^2+\frac{1}{\gamma^2}\|\Delta_x^t\|^2+\|\Delta_w^t\|^2+\frac{1}{\lambda^2}\|\Delta_y^t\|^2\big]\nonumber\\
\leq & \mathbb{E}\left[8\|\Delta_v^t\|^2+\frac{1}{\gamma^2}\|\Delta_x^t\|^2+\frac{2\beta_{t+1}}{\mu^2\eta_t^2}\|\Delta_w^t\|^2+\frac{4L^2}{\mu\lambda}\|\Delta_y^t\|^2\right]\nonumber\\
\leq & 4\bigg(\frac{(\phi_t-\phi_{t+1})}{\gamma\eta_T}+\frac{192L^4 r^2\eta_t}{\mu^2\eta_T} +\frac{45L^2d^2r^2}{ q\eta_t\eta_T\mu^2} +\frac{4c^2\eta_t^3\sigma^2}{q\mu^2\eta_T}\bigg).
\end{align}

\paragraph{Step 3: Take the telescoping sum of the above bound.} Using the result in Lemma \ref{lemma:measure_bound} and taking the summation of \eqref{eq:sum1} over $t=1,2,\cdots,T$, we have
\begin{align}\label{eq:sum_2}
& \frac{1}{T}\sum_{t=1}^T\mathbb{E}\left[\left\|\mathfrak{g}(x_t,y_t)\right\|^2\right]\leq 7L^2r^2+4\mu^2D_y^2\nonumber\\
&+ \underbrace{\frac{16}{T}\sum_{t=1}^T\frac{(\phi_t-\phi_{t+1})}{\gamma\eta_T}}_{T_1}+\underbrace{\frac{16}{T}\sum_{t=1}^T\frac{192L^4 r^2\eta_t}{\mu^2\eta_T}}_{T_2} +\underbrace{\frac{16}{T}\sum_{t=1}^T\frac{45L^2d^2 r^2}{ q\eta_t\eta_T\mu^2}}_{T_3} +\underbrace{\frac{16}{T}\sum_{t=1}^T\frac{4c^2 \eta_t^3\sigma^2}{  q\mu^2\eta_T}}_{T_4}.
\end{align}
\begin{lemma}\label{lemma:bound_T}
$T_1+T_2+T_3+T_4\leq \frac{R_1}{T} (m+T)^{\frac{11-3e}{15-3e}}
+ \frac{R_2\ln(m+T)}{T}(m+T)^{\frac{9-e}{15-3e}}+\frac{720d}{q}(m+T)^{-\frac{4}{15-3e}}$.
\end{lemma}
Using the parameter setting for $\mu$ and $r$ in Theorem \ref{thm:Opt}, we can get $7 L^2r^2\leq \frac{7}{d}(m+T)^{-\frac{4}{3}}$ and $4\mu^2D_y^2\leq 4D_y^2 (m+T)^{-\frac{4}{15-3e}}$, which leads to $7L^2r^2+4\mu^2D_y^2\leq (4D_y^2+\frac{7}{d})(m+T)^{-\frac{4}{15-3e}}$.
Substituting this result and the one in Lemma \ref{lemma:bound_T} into \eqref{eq:sum_2}, we can derive the final conclusion.\Halmos
\end{proof}
By the expressions of $R_1, R_2, R_3$ and $q = O(d\mu^{-e})$, it follows that $R_1, R_2, R_3$ are all $O(1)$. Hence, the convergence bound in Theorem \ref{thm:Opt} can be expressed as
\[
\min_{t\leq T} \mathbb{E}\left[\left\|\mathfrak{g}(x_t,y_t)\right\|^2\right]\leq O\big(T^{-\frac{4}{15-3e}}\ln(m+T)\big).
\]
This implies that MGS converges to a stationary point of \eqref{eq:minmax_problem} at the rate $\tilde{O}\big(T^{-\frac{4}{15-3e}}\big)$. To the best of our knowledge, this is the first finite-time convergence guarantee for single-loop algorithms to solve CSO problems \eqref{eq:constrained_problem}. 

To determine the sample complexity needed to attain a given error level, note from the above bound that $\tilde{O}\big(\epsilon^{-\frac{15-3e}{4}}\big)$ iterations are needed to ensure $\min_{t\leq T} \mathbb{E}\left[\left\|\mathfrak{g}(x_t,y_t)\right\|^2\right]\leq \epsilon$ for any $\epsilon>0$. Since each iteration requires $O(q)=O(d\mu^{-e})$ simulations, the total number of simulations is $T\cdot O(q)=T\cdot O(d\epsilon^{-\frac{e}{2}})=\tilde{O}\big(d\epsilon^{-\frac{15-e}{4}}\big)$.
Moreover, by tuning the parameter $e$, we can control the convergence rate. In particular, choosing $e=1$ yields the best convergence rate $\tilde{O}\big(T^{-\frac{1}{3}}\big)$ and the best simulation complexity $\tilde{O}(d\epsilon^{-3.5})$. These findings are summarized in the following corollary.

\begin{corollary}\label{coro:complexity}
Suppose that the conditions in Theorem \ref{thm:Opt} hold and take $e=1$. Then MGS attains a convergence rate of $\tilde{O}(T^{-\frac{1}{3}})$, and $\tilde{O}(d\epsilon^{-3.5})$ simulations are required to obtain an expected $\epsilon$-stationary point.
\end{corollary}
Corollary \ref{coro:complexity} indicates that increasing the number of gradient replications can yield more accurate gradient estimates, which can improve both iteration and sample complexities. However, as demonstrated in our numerical experiments, a relatively small value of $q$ is often sufficient to achieve substantial performance gains.
% To see the derivation of Corollary \ref{coro:complexity}, it is obvious from Theorem \ref{thm:Opt} that $O(\epsilon^{-3})$ iterations are required to achieve $\min_{t\leq T} \mathbb{E}\left[\left\|\mathfrak{g}(x_t,y_t)\right\|^2\right]\leq \epsilon$ for any $\epsilon>0$. Due to that $O(q)=O(d\mu^{-1})$ simulations are required for each iteration, the total number of simulations is $T\cdot O(q)=T\cdot O(d\epsilon^{-\frac{1}{2}})=O\left(\frac{d}{\epsilon^{3.5}}\right)$.
% $\sigma^2=O(d)$. Accordingly, we have $T_1=O(T^{-\frac{1}{3}}\cdot \frac{d}{q})$, and $T_4=O(T^{-\frac{1}{3}}\cdot \frac{d}{q})$. When $q=d$, to get $O(T^{-\frac{1}{3}})\leq \epsilon$ for some $\epsilon>0$, we need $T\geq O(1/\epsilon^3)$, and $\mu\leq \epsilon^{1/2}$. 

\subsection{Interpretation and Implications}\label{subsec:interpretation}
Here, we discuss our results in terms of comparison with prior work, effect of CRN, and implications for problem \eqref{eq:constrained_problem}.

\textbf{Comparison with prior work.} The convergence rate of $\tilde{O}(T^{-\frac{1}{3}})$ in Corollary \ref{coro:complexity} for solving \eqref{eq:constrained_problem} seems worse than $O(T^{-\frac{1}{2}})$ achieved by SA for solving unconstrained SO in the form of $\min_{x\in\mathcal{X}}h_0(x)$ \citep{fu2015handbook}. The degradation arises from the dual update used to handle constraints on performance measures, which leads to the coupling between primal and dual dynamics. MGS must balance the descent of the objective with the ascent of the penalty term, and this coupling may slow down convergence. Besides, unlike other random search-based algorithms whose convergence rates are usually dimension-dependent, such as $O(T^{-\frac{1}{d+2}})$ in \cite{wang2025gaussian} for solving $\min_{x\in\mathcal{X}}h_0(x)$, the convergence rate of our algorithm does not heavily depend on the problem dimension, making it well-suited for high-dimensional problems.

% \begin{remark}
% The results in Theorem \ref{thm:Opt} and Corollary \ref{coro:complexity} provide the convergence guarantees to a stationary point of problem \eqref{eq:minmax_problem}. As for problem \eqref{eq:constrained_problem}, one can show that $x$ is a critical KKT point (see its definition in \cite{boob2023stochastic}) of \eqref{eq:constrained_problem} when $(x,y)$ is a stationary point satisfying $\mathfrak{g}(x,y)=0$ and $y< \overline{y}$. Specifically, $\mathfrak{g}_x(x,y)=0$ implies the stationarity condition; $\mathfrak{g}_y(x,y)=0$ and $y < \overline{y}$ implies complementary slackness and primal feasibility.
% \end{remark}

Besides our method, \cite{nguyen2023stochastic} proposed a nested-loop algorithm for nonconvex problems with stochastic functional constraints, which offers the sample complexity bound of $O\big(d\epsilon^{-3}\big)$. However, their bound is only stated for the weighted average of all iterates, i.e., $\Bar{x}=(\sum_t \gamma_t)^{-1}\sum_t \gamma_t x_t$ instead of directly for $\{x_t\}_{t\geq 0}$ as in our results. Moreover, the nested-loop structure calls an inner-loop algorithm to approximately solve a convex relaxation of \eqref{eq:constrained_problem} at each outer iteration. The number of outer iterations depends on the solution error of the inner loop. This makes the practical implementation and parameter tuning of nested-loop algorithms more difficult. In contrast, single-loop algorithms update all variables within a unified iterative procedure, without additional inner subproblems or solvers.

Compared with the literature on zeroth-order min-max optimization, our work also achieves the best-known sample complexity bound in stochastic, nonconvex-concave settings. In particular, the previous best-known bound of $O(d\epsilon^{-4})$ was established in \cite{NEURIPS2024_413885e7} using an extra-gradient framework, while our method and analysis improve this bound to $\tilde{O}(d\epsilon^{-3.5})$.

% In contrast, \cite{nguyen2023stochastic} propose a nested-loop algorithm for nonconvex problems with stochastic functional constraints, which offers the sample complexity bound of $O\big(\frac{d}{\epsilon^3}\big)$. However, their bound is only stated for the weighted average of all iterates, i.e., $\Bar{x}=(\sum_t \gamma_t)^{-1}\sum_t \gamma_t x_t$ instead of directly for $\{x_t\}_{t\geq 0}$ as in our results. Moreover, the nested-loop structure calls an inner-loop algorithm to approximately solve a convex relaxation of \eqref{eq:constrained_problem} at each outer iteration. The number of outer iterations depends on the solution error of the inner loop. This makes the practical implementation and parameter tuning of nested-loop algorithms more difficult.

\textbf{Effect of CRN.} Here, we investigate the effect of CRN on the convergence bound of MGS. We denote the bound in Theorem \ref{thm:Opt} as a function of $T$ and $\sigma^2$, i.e.,
$$\mathcal{B}(T;\sigma^2)=\frac{R_1(\sigma^2)}{T} (m+T)^{\frac{11-3e}{15-3e}}
+ \frac{R_2(\sigma^2)\ln(m+T)}{T}(m+T)^{\frac{9-e}{15-3e}}+R_3(m+T)^{-\frac{4}{15-3e}},$$
where $R_1(\sigma^2)=\frac{16}{k_1}\big(F_r(x_1)+\frac{8L^2k_1\mu}{k_2} D_y^2+ \frac{2k_1\mu^{1-e}\sigma^2}{\eta_{0}q}
-F_r^*\big)$ and $R_2(\sigma^2)=\frac{64c^2\sigma^2}{q}+\frac{3200L^2}{d}$. Therefore, $\mathcal{B}(T;\sigma^2)$ can be written as a linear function of $\sigma^2$:
$$\mathcal{B}(T;\sigma^2)=\mathcal{B}(T;0)+\frac{B(T)}{q}\sigma^2,$$
where $B(T)=\frac{32\mu^{1-e}}{\eta_0 T}(m+T)^{\frac{11-3e}{15-3e}}+\frac{64c^2\ln(m+T)}{T}(m+T)^{\frac{9-e}{15-3e}}$.
Then, we can denote the convergence bounds for implementing CRN and independent randomness as $\mathcal{B}(T;\sigma_{\mathrm{CRN}}^2)$ and $\mathcal{B}(T;\sigma_{\mathrm{IND}}^2)$, respectively. Therefore, we have the following corollary.

\begin{corollary}
The bounds of MGS under CRN and independent randomness for gradient estimations satisfy
$$\frac{\mathcal{B}(T;\sigma_{\mathrm{IND}}^2)-\mathcal{B}(T;0)}{\mathcal{B}(T;\sigma_{\mathrm{CRN}}^2)-\mathcal{B}(T;0)}=\frac{\sigma_{\mathrm{IND}}^2}{\sigma_{\mathrm{CRN}}^2}.$$
\end{corollary}

Besides, since the variance-related part of the bound enters only through $\sigma^2/q$, implementing CRN has the same effect as increasing the number of gradient replications from $q$ to $q\sigma_{\mathrm{IND}}^2/\sigma_{\mathrm{CRN}}^2$. Thus, CRN yields a constant-factor reduction in the replication budget required to attain a target stationarity level, without changing the asymptotic rate exponent.

\textbf{Implication for problem \eqref{eq:constrained_problem}.} While $\|\mathfrak g(x,y)\|$ is a standard stationarity measure for the min-max formulation \eqref{eq:minmax_problem}, our ultimate goal is solving the original CSO problem \eqref{eq:constrained_problem}.
We therefore introduce a KKT-style residual and connect it to $\mathfrak g(x,y)$.
For the original problem \eqref{eq:constrained_problem}, define the KKT gap at a primal-dual pair $(x,y)\in\mathcal X\times\mathcal Y$ as
\[
\mathrm{KKT}(x,y):=\big\|\frac{1}{\gamma}\left(x-\mathcal{P}_{\mathcal X}[x-\gamma \nabla_x h(x,y)]\right)\big\|
+\max_{j\in\mathcal I}[h_j(x)]_+
+\max_{j\in\mathcal I} y(j) |h_j(x)|.
\]
The first term matches the primal component of $\mathfrak g(x,y)$; the latter two terms quantify primal infeasibility and complementarity gap.

The following lemma demonstrates that controlling $\|\mathfrak g(x,y)\|$ while also preventing dual saturation leads to a small approximate KKT gap.

\begin{lemma}[Bound for KKT Gap]\label{lemma:kkt_bridge}
Assume Assumptions \ref{assump:bounded}--\ref{assump:bounded_variance} hold. Let $(x,y)\in\mathcal X\times\mathcal Y$ be the iterate derived by MGS satisfying $\|\mathfrak g(x,y)\|\le \epsilon$ and, additionally,
$y<\bar y$. Then there exists a constant $C>0$ such that
\[
\mathrm{KKT}(x,y)
\;\le\; C\,\epsilon.
\]
\end{lemma}
Lemma \ref{lemma:kkt_bridge} shows that, under proper conditions, the KKT gap can converge at the same rate as $\|\mathfrak{g}(x,y)\|$. Specifically, any point $(x,y)\in\mathcal{X}\times \mathcal{Y}$ satisfying $\|\mathfrak{g}(x,y)\|=0$ and $0\leq y <\bar y$ is a KKT point of the original CSO problem.
\begin{proof}[Proof of Lemma \ref{lemma:kkt_bridge}]
Define
$
\delta:=\min_{j\in\mathcal I}\bigl(\bar{y}(j)-y(j)\bigr)$, and $y_{\max}:=\max_{j\in\mathcal I}\bar{y}(j).$
Since \(y<\bar{y}\) component-wise, we have \(\delta>0\). By the definition of the projected gradient,
$\left\|\frac{1}{\gamma}\left(x-\mathcal{P}_{\mathcal X}[x-\gamma \nabla_x h(x,y)]\right)\right\|
= \|\mathfrak{g}_x(x,y)\|
\le \|\mathfrak{g}(x,y)\|
\le \epsilon$.
Thus, it remains to bound the primal infeasibility and complementarity terms by
\(\|\mathfrak{g}_y(x,y)\|\).

For each \(j\in\mathcal I\), define $\hat{y}(j):=y(j)+\lambda h_j(x)$, and $
\mathcal{P}_j[\cdot]:=\mathcal{P}_{[0,\bar{y}(j)]}[\cdot]$.
Since \(\mathcal{Y}=[0,\bar{y}]\), the projection is performed component-wise, and hence $(\mathfrak{g}_y(x,y))_j=\frac{1}{\lambda}\bigl(y(j)-\mathcal{P}_j[\hat{y}(j)]\bigr)$.

Then, we consider three cases.

\textit{Case 1:} \(0\le \hat{y}(j)\le \bar{y}(j)\). In this case, $\mathcal{P}_j[\hat{y}(j)]=\hat{y}(j)$, so $(\mathfrak{g}_y(x,y))_j
=
\frac{1}{\lambda}\bigl(y(j)-\hat{y}(j)\bigr)
=
-h_j(x).$
Therefore, we have
\begin{align*}
[h_j(x)]_+ \le |h_j(x)| &= |(\mathfrak{g}_y(x,y))_j|,\\
y(j)|h_j(x)|
\le \bar{y}(j)\,|(\mathfrak{g}_y(x,y))_j|
&\le y_{\max}|(\mathfrak{g}_y(x,y))_j|.
\end{align*}

\textit{Case 2:} \(\hat{y}(j)<0\). In this case, \(\mathcal{P}_j[\hat{y}(j)]=0\), so $|(\mathfrak{g}_y(x,y))_j|=\frac{1}{\lambda}y(j)$.
Moreover, \(\hat{y}(j)=y(j)+\lambda h_j(x)<0\) implies \(h_j(x)<0\), and hence $[h_j(x)]_+ = 0.$
By Assumption \ref{assump:lipschitz}, $\|\nabla_y h(x,y)\|\le \Lambda$ since $h(x,y)$ is differentiable and Lipschitz continuous. Therefore, $|h_j(x)|\leq \Lambda$ and
\[
y(j)|h_j(x)|
\le \lambda\Lambda\,|(\mathfrak{g}_y(x,y))_j|.
\]

\textit{Case 3:} \(\hat{y}(j)>\bar{y}(j)\). In this case, we have \(\mathcal{P}_j[\hat{y}(j)]=\bar{y}(j)\), and thus $|(\mathfrak{g}_y(x,y))_j|=\frac{\bar{y}(j)-y(j)}{\lambda}\ge \frac{\delta}{\lambda}$.
Since \(\hat{y}(j)=y(j)+\lambda h_j(x)>\bar{y}(j)\), we have $h_j(x)>\frac{\bar{y}(j)-y(j)}{\lambda}\ge\frac{\delta}{\lambda}>0$.
Therefore, \([h_j(x)]_+=h_j(x)\). Using Assumption \ref{assump:lipschitz} again, we can get
$$[h_j(x)]_+
\le \Lambda
\le \frac{\Lambda\lambda}{\delta}|(\mathfrak{g}_y(x,y))_j|,\quad y(j)|h_j(x)|
\le y_{\max}\Lambda
\le \frac{y_{\max}\Lambda\lambda}{\delta}|(\mathfrak{g}_y(x,y))_j|.$$

Now set $M:=\max\left\{1,\frac{\Lambda\lambda}{\delta}\right\}$. Since \(0\le y(j)\le \bar{y}(j)\le y_{\max}\) for all \(j\in\mathcal I\), we have
\(\delta\le y_{\max}\), and therefore \(\lambda\Lambda\le y_{\max}M\). Combining the
three cases yields, for every \(j\in\mathcal I\),
\[
[h_j(x)]_+ \le M\,|(\mathfrak{g}_y(x,y))_j|,
\qquad
y(j)|h_j(x)| \le y_{\max}M\,|(\mathfrak{g}_y(x,y))_j|.
\]
Taking the maximum over \(j\) and using \(\|v\|_\infty\le \|v\|\), we obtain
\[
\max_{j\in\mathcal I}[h_j(x)]_+
\le M\|\mathfrak{g}_y(x,y)\|,
\qquad
\max_{j\in\mathcal I} y(j)|h_j(x)|
\le y_{\max}M\|\mathfrak{g}_y(x,y)\|.
\]

Therefore,
$$\mathrm{KKT}(x,y)\le \|\mathfrak{g}_x(x,y)\|
+ M\|\mathfrak{g}_y(x,y)\|
+ y_{\max}M\|\mathfrak{g}_y(x,y)\|\le \Bigl(1+(1+y_{\max})M\Bigr)\epsilon.$$
% \begin{align*}
% \mathrm{KKT}(x,y)
% &\le \|\mathfrak{g}_x(x,y)\|
% + M\|\mathfrak{g}_y(x,y)\|
% + y_{\max}M\|\mathfrak{g}_y(x,y)\| \\
% &\le \Bigl(1+(1+y_{\max})M\Bigr)\|\mathfrak{g}(x,y)\| \\
% &\le \Bigl(1+(1+y_{\max})M\Bigr)\epsilon.
% \end{align*}
Hence the claim holds with $C:=1+(1+y_{\max})\max\left\{1,\frac{\Lambda\lambda}{\delta}\right\}$. \Halmos
\end{proof}

\section{Numerical Experiments}\label{sec:numerical}
In this section, we conduct numerical experiments on a serial queuing system and a high-dimensional stochastic optimization problem to validate the effectiveness of our algorithm and theoretical results. The detailed experiment settings can be found in the Appendix.

\subsection{Serial Queuing System}
We consider a serial queuing system composed of five subqueues arranged in sequence, each operated by a single server. Our objective is to minimize the overall service cost by appropriately selecting the service rates of all servers, while ensuring a specified level of service quality. Specifically, the expected waiting time should remain below a given threshold. The CSO problem is formulated as
\begin{align*}
\min_{x\in\mathcal{X}}\; & h_0(x)=c^T x\\
\text{s.t.}\; & h_1(x)=\mathbb{E}[W(x;\xi)]- \overline{w}\leq 0,
\end{align*}
where $x\in\mathcal{X}\subset\mathbb{R}^5$ is the vector of service rates. $h_0: \mathcal{X}\to \mathbb{R}$ is the cost function and $c\in\mathbb{R}^5$ is the vector of cost coefficients. $W(\cdot;\xi): \mathcal{X} \to \mathbb{R}$ is the total waiting time of the queuing system under the sample path $\xi$, and $\overline{w}$ is the specified maximum waiting time.

First, we test MGS under different numbers of replications, i.e., $q$, for each gradient estimation. We run each algorithm with different initial solutions and randomness 20 times, and the results are shown in Figure \ref{fig:mgs_performance_b}. The dark lines and shaded areas represent the average performance and standard deviation, respectively. It shows that increasing $q$ substantially improves the stability of the optimization trajectories. Compared with $q=1$,  the cases with $q=10, 20$, and $50$ exhibit much faster objective reduction, quicker adjustment of the dual variable, and significantly smaller oscillations in both the constraint value and the stationarity measure. At the same time, the difference among $q=10,20,$ and $50$ is relatively small. This observation is consistent with the variance-reduction effect of batched gradient estimation: larger $q$ leads to more accurate gradient estimates, but in practice, a small $q$ is often sufficient to obtain substantial gains. The stationarity measure does not converge to zero exactly because it is evaluated by simulation and contains estimation error.
\begin{figure}[htpb]
    \centering
    \includegraphics[width=0.8\linewidth]{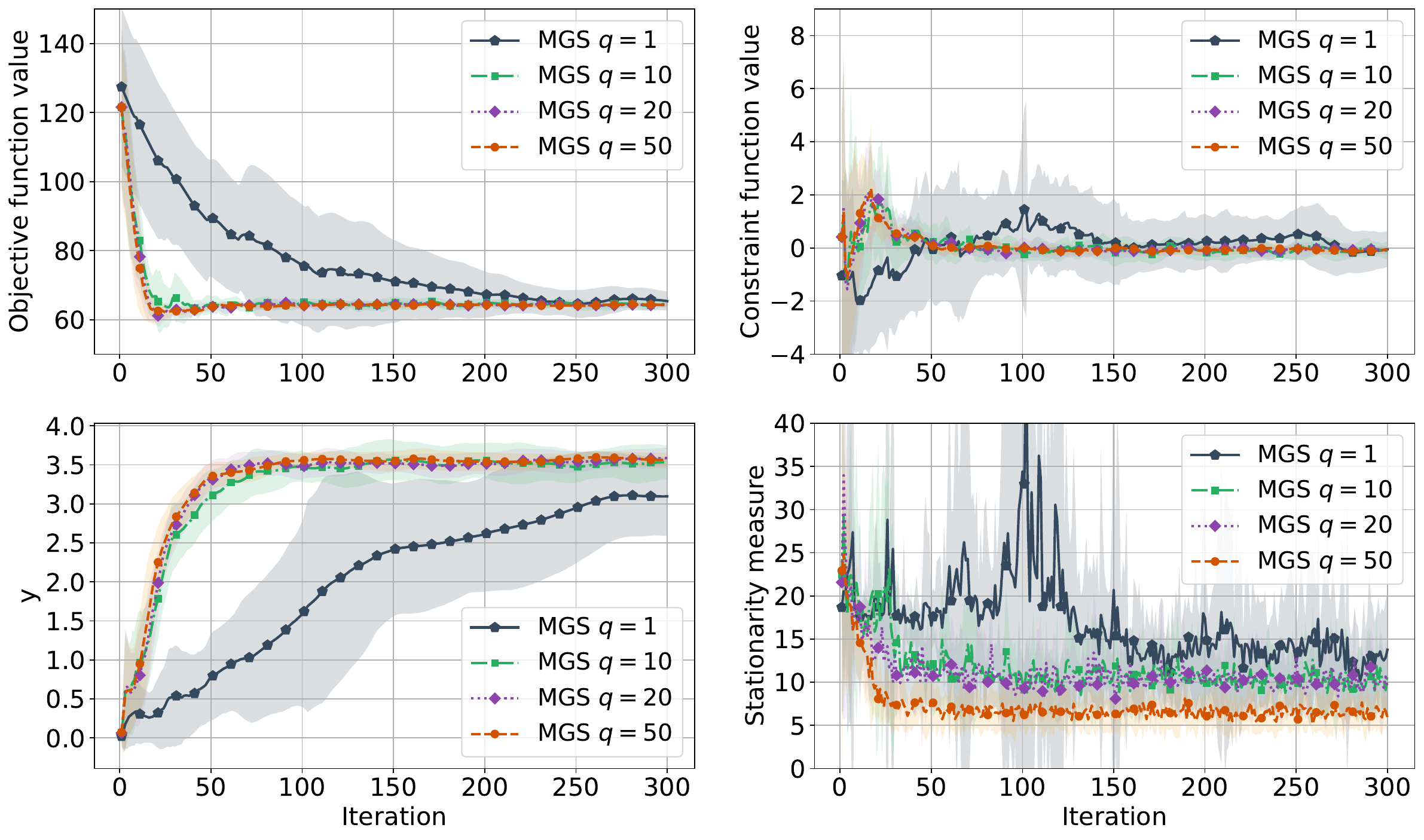}
    \caption{Performance of MGS with different $q$ in the queuing system}
    \label{fig:mgs_performance_b}
\end{figure}

\begin{figure}[htpb]
    \centering
    \includegraphics[width=0.85\linewidth]{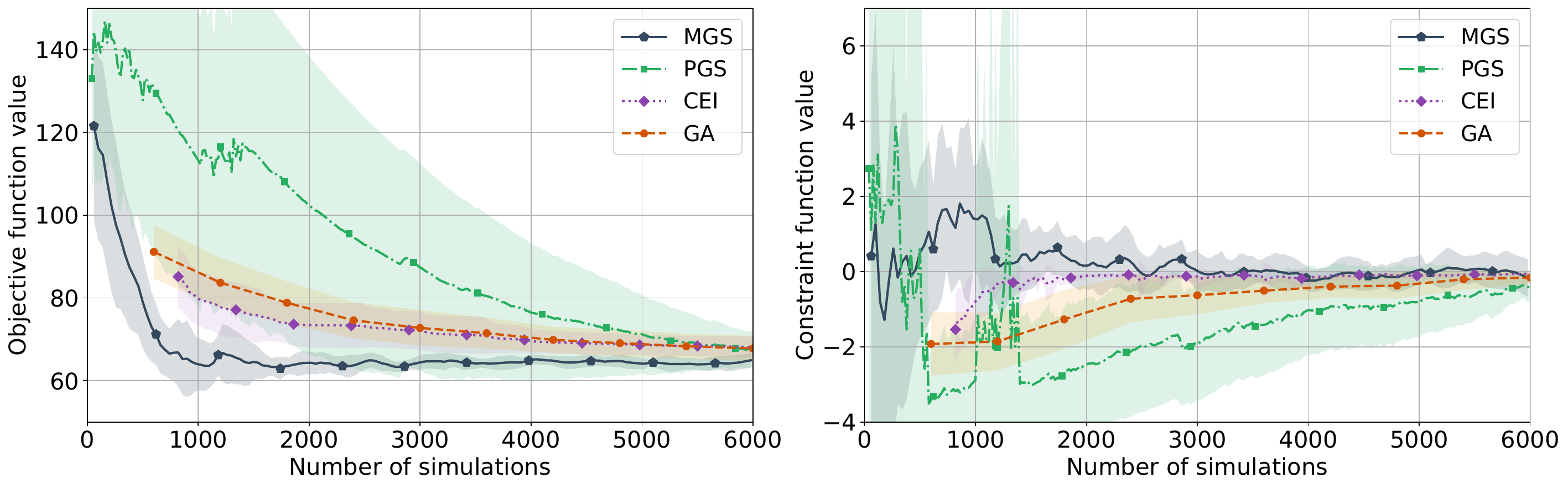}
    \caption{Performance of different algorithms in the queuing system}
    \label{fig:queue_compare}
\end{figure}

We also compare MGS with three other SO algorithms that can solve problem \eqref{eq:constrained_problem}. The first is the penalty-based gradient search (PGS) algorithm \citep{wang2008stochastic}, which incorporates constraint violations into the objective function to be minimized. The second is a BO algorithm using constrained expected improvement (CEI) \citep{gardner2014bayesian}, a method commonly employed for CSO problems. The third is the genetic algorithm (GA) \citep{rajwar2023exhaustive}, a meta-heuristic optimization method. Their performance as a function of the number of simulations is presented in Figure \ref{fig:queue_compare}.
 The results indicate that all methods eventually converge to feasible solutions; however, MGS is significantly more efficient than the others. It converges faster and yields a better final solution, characterized by a lower objective value and approximately zero constraint violation. In contrast, the other methods approach approximately zero constraint violation from the negative side, meaning their constraints are overly satisfied, which leads to overly conservative solutions.
% The figure shows that MGS can converge to a feasible solution with varying choices of $q$. Note that the stationarity measures of MGS do not converge to 0, which is largely due to the estimation error (the measure can only be calculated by sampling). Furthermore, the performance of MGS $(q=10)$ is very close to that of MGS $(q=50)$, which shows that only a few replications for each gradient estimation can significantly improve the performance.
% Compared with MGS, CEI can find solutions with higher objective function values and higher computational complexities, although the constraint violation is satisfactory. For PGS, the final cost is slightly higher than that of others, as the constraint is oversatisfied. GA can find high-quality solutions
% with lower variance, although the computational complexity is higher.

\subsection{Ablation Study}
To make the methodological novelty explicit and to diagnose which design elements matter in CSO, we conduct an ablation study that removes one component at a time from MGS, including CRN, momentum updates, and strong concavity regularization, while keeping the other components. Specifically, we test three more variants from MGS: MGS with no CRN (MGS (NCRN)), MGS with no momentum updates (MGS (NMU)), and MGS without strong concavity (MGS $\mu=0$), and compare them with standard MGS.

\begin{figure}[htb]
    \centering
    \includegraphics[width=0.8\linewidth]{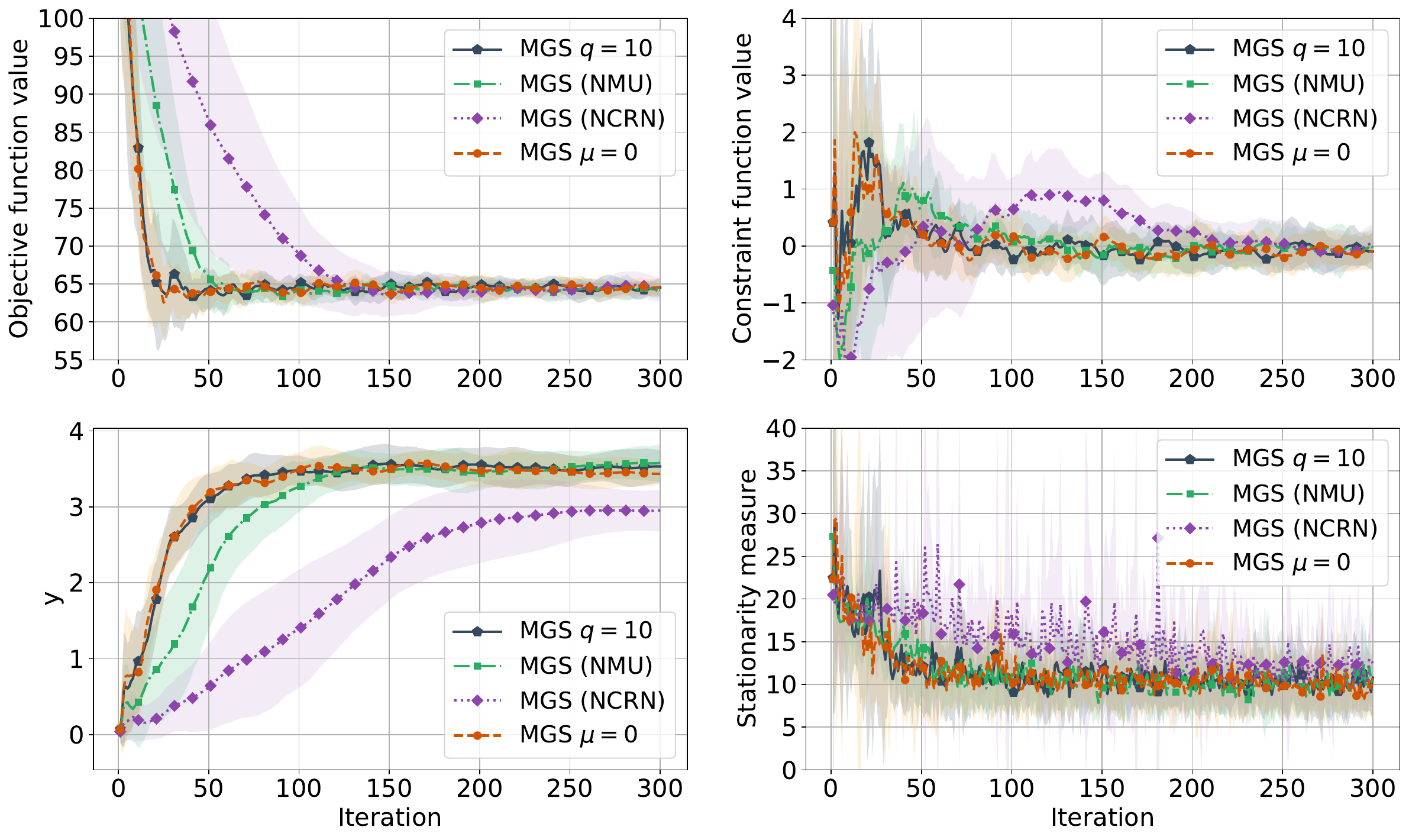}
    \caption{Performance of MGS under different settings in the queuing system}
    \label{fig:ablation_compare}
\end{figure}

The performance comparison results are summarized in Figure \ref{fig:ablation_compare}. It shows that removing CRN implementation can substantially diminish the convergence speed of MGS, which is consistent with our analysis in Sections \ref{subsec:crn_analysis} and \ref{subsec:interpretation}: implementing CRN can reduce the variance of the gradient estimations. Although CRN can only improve our finite-time bound by a variance constant, it plays an important role in improving the empirical performance. Similar degradation is also observed for MGS with no momentum updates, but it is smaller. In contrast, setting $\mu=0$ for MGS yields nearly identical performance to the standard one. We emphasize that the strong-concavity regularization term mainly serves as a stability mechanism to prevent cycling behavior of GDA in general min-max problems (see Section \ref{subsec:mgs_algo}), and it is also essential for establishing our non-asymptotic results.

\subsection{High-Dimensional Case}
We consider a problem with a 2000-dimensional decision space to verify the capability of MGS to deal with high-dimensional problems. Specifically, we consider the CSO problem:
\begin{equation}
\begin{aligned}
\min_{x\in\mathcal{X}\subseteq \mathbb{R}^{2000}} & h_0(x)=\mathbb{E}[h_0(x;a)]=\mathbb{E}[\langle a, x^3\rangle] -5x^Tx\\
s.t.\; & h_1(x)=\mathbb{E}[h_1(x;b)]=\mathbb{E}[\langle b,x^2\rangle]-200 \leq 0,
\end{aligned}
\end{equation}
where $a,b\in\mathbb{R}^{2000}$ are random variables following Gaussian distributions.

Due to the limited scalability of CEI and GA in high-dimensional settings, we only evaluate MGS and PGS on this problem. The results are presented in Figure \ref{fig:function_compare}. These results indicate that MGS converges to feasible solutions for both $q=5$ and $q=10$, although the finite-sample performance for $q=10$ is slightly inferior to that for $q=5$. This implies that using a larger $q$ may reduce efficiency in terms of sample complexity, even though it can improve the quality of gradient estimation. By comparison, PGS converges much more slowly and exhibits a small but positive constraint violation. Overall, the results show that MGS scales well to high-dimensional CSO problems and achieves superior performance.
\begin{figure}[htpb]
    \centering
    \includegraphics[width=0.85\linewidth]{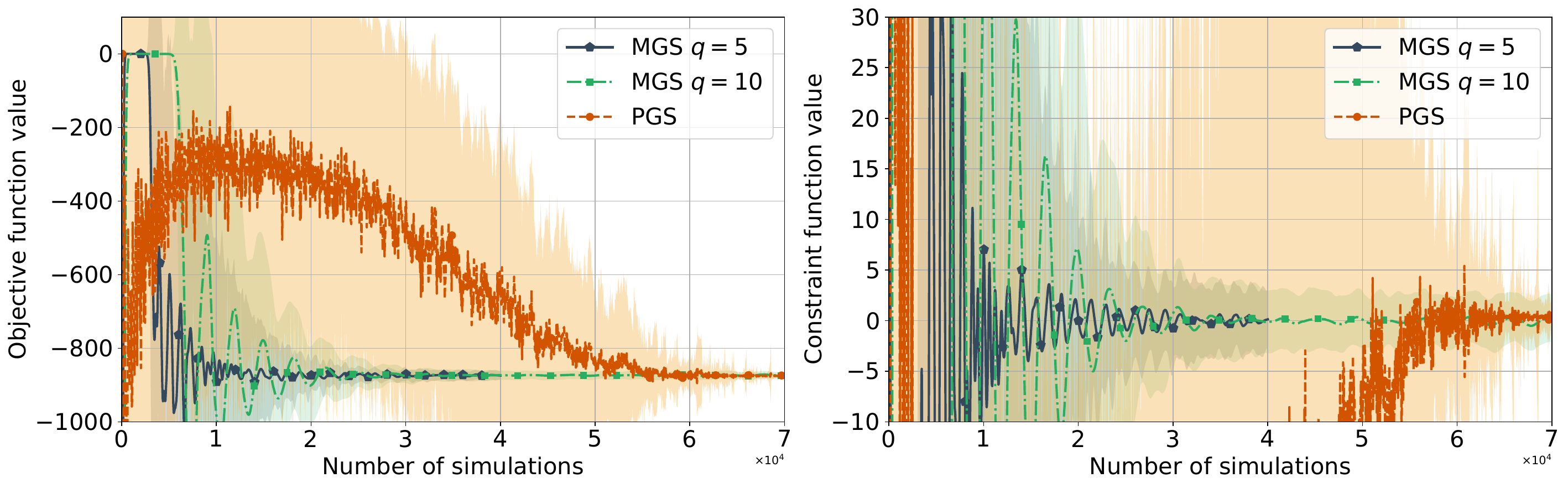}
    \caption{Performance of different algorithms in the high-dimensional case}
    \label{fig:function_compare}
\end{figure}

\section{Conclusion}\label{sec:conclusion}
This paper studies a continuous simulation optimization problem with stochastic and non-analytical constraints. We reformulate it as a min-max problem and develop a min-max gradient search (MGS) algorithm that integrates Gaussian-smoothing finite-difference gradient estimators within a primal-dual scheme to solve it. The algorithm leverages common random numbers and momentum-based updates to accelerate convergence. On the theoretical side, we establish the first non-asymptotic convergence guarantee for single-loop CSO algorithms, showing that MGS can converge to a stationary point at a rate of $\tilde{O}(T^{-1/3})$. Empirically, MGS is able to obtain feasible, high-quality solutions on both a queuing system and a 2000-dimensional problem, and outperforms existing methods in terms of both convergence rate and scalability.

Beyond the theoretical guarantees, this work also highlights the potential of min-max gradient search as a practically implementable framework for CSO. In simulation environments, both the objective and the constraints are available only through noisy black-box outputs, so the key challenge is not just to optimize performance, but to systematically balance performance improvement and feasibility under limited simulation budgets. From this perspective, the min-max formulation is particularly appealing because it incorporates constraint control directly into the search dynamics, rather than relying on carefully tuned penalty parameters or repeatedly solving auxiliary inner problems. This makes the proposed approach not only theoretically grounded, but also well aligned with the structural needs of CSO. We hope this viewpoint can help broaden the algorithmic toolbox for simulation optimization and motivate further research on scalable methods for more realistic constrained systems.

%\THEEndNotes
% \begingroup \parindent 0pt \parskip 0.0ex \def\enotesize{\normalsize} \theendnotes \endgroup

% Acknowledgments here
% \ACKNOWLEDGMENT{We would like to express our sincere gratitude to [acknowledge individuals, organizations, or institutions] for their invaluable contributions to this research. We are also grateful to [mention any additional acknowledgements, such as technical assistance, data providers, or colleagues] for their support and assistance throughout the course of this work.}

% References here (outcomment the appropriate case)

% CASE 1: BiBTeX used to constantly update the references
%   (while the paper is being written).

\bibliographystyle{apalike}
\bibliography{reference}

\begin{thebibliography}{}

\bibitem[Amaran et~al., 2016]{amaran2016simulation}
Amaran, S., Sahinidis, N.~V., Sharda, B., and Bury, S.~J. (2016).
\newblock Simulation optimization: a review of algorithms and applications.
\newblock {\em Annals of Operations Research}, 240(1):351--380.

\bibitem[An et~al., 2024]{NEURIPS2024_413885e7}
An, W., Liu, Y., Shang, F., and Liu, H. (2024).
\newblock Robust and faster zeroth-order minimax optimization: Complexity and applications.
\newblock In {\em Proceedings of the 38th International Conference on Neural Information Processing Systems}, volume~37, pages 37050--37069.

\bibitem[Bassamboo and Randhawa, 2010]{bassamboo2010accuracy}
Bassamboo, A. and Randhawa, R.~S. (2010).
\newblock On the accuracy of fluid models for capacity sizing in queueing systems with impatient customers.
\newblock {\em Operations Research}, 58(5):1398--1413.

\bibitem[Berahas et~al., 2022]{berahas2022theoretical}
Berahas, A.~S., Cao, L., Choromanski, K., and Scheinberg, K. (2022).
\newblock A theoretical and empirical comparison of gradient approximations in derivative-free optimization.
\newblock {\em Foundations of Computational Mathematics}, 22(2):507--560.

\bibitem[Bhatnagar et~al., 2011]{bhatnagar2011stochastic}
Bhatnagar, S., Hemachandra, N., and Mishra, V.~K. (2011).
\newblock Stochastic approximation algorithms for constrained optimization via simulation.
\newblock {\em ACM Transactions on Modeling and Computer Simulation (TOMACS)}, 21(3):1--22.

\bibitem[Cakmak et~al., 2021]{cakmak2021solving}
Cakmak, S., Wu, D., and Zhou, E. (2021).
\newblock Solving bayesian risk optimization via nested stochastic gradient estimation.
\newblock {\em IISE Transactions}, 53(10):1081--1093.

\bibitem[Chang et~al., 2025]{chang2025maximizing}
Chang, P.-C., Yeh, C.-T., Yu, V.~F., and Chiu, S.-F. (2025).
\newblock Maximizing network reliability through simulation optimization: a component configuration strategy.
\newblock {\em Annals of Operations Research}, pages 1--36.

\bibitem[Cordoba-Pacheco and Ruiz, 2024]{cordoba2024optimal}
Cordoba-Pacheco, A. and Ruiz, F. (2024).
\newblock Optimal energy management in multi-microgrids. a scenario-based mpc approach.
\newblock In {\em 2024 European Control Conference (ECC)}, pages 3709--3714. IEEE.

\bibitem[Cutkosky and Orabona, 2019]{cutkosky2019momentum}
Cutkosky, A. and Orabona, F. (2019).
\newblock Momentum-based variance reduction in non-convex sgd.
\newblock In {\em Proceedings of the 33rd International Conference on Neural Information Processing Systems}, volume~32, pages 15236--15245.

\bibitem[Eriksson and Poloczek, 2021]{pmlr-v130-eriksson21a}
Eriksson, D. and Poloczek, M. (2021).
\newblock Scalable constrained bayesian optimization.
\newblock In {\em Proceedings of The 24th International Conference on Artificial Intelligence and Statistics}, volume 130, pages 730--738. PMLR.

\bibitem[Fan and Hu, 2018]{fan2018surrogate}
Fan, Q. and Hu, J. (2018).
\newblock Surrogate-based promising area search for lipschitz continuous simulation optimization.
\newblock {\em INFORMS Journal on Computing}, 30(4):677--693.

\bibitem[Fu et~al., 2015]{fu2015handbook}
Fu, M.~C. et~al. (2015).
\newblock {\em Handbook of simulation optimization}, volume 216.
\newblock Springer.

\bibitem[Gardner et~al., 2014]{gardner2014bayesian}
Gardner, J.~R., Kusner, M.~J., Xu, Z.~E., Weinberger, K.~Q., and Cunningham, J.~P. (2014).
\newblock Bayesian optimization with inequality constraints.
\newblock In {\em Proceedings of the 31st International Conference on International Conference on Machine Learning}, volume~32, pages 937--945.

\bibitem[Hern{\'a}ndez-Lobato et~al., 2016]{hernandez2016general}
Hern{\'a}ndez-Lobato, J.~M., Gelbart, M.~A., Adams, R.~P., Hoffman, M.~W., and Ghahramani, Z. (2016).
\newblock A general framework for constrained bayesian optimization using information-based search.
\newblock {\em Journal of Machine Learning Research}, 17(160):1--53.

\bibitem[Hu and Fu, 2025]{hu2025convergence}
Hu, J. and Fu, M.~C. (2025).
\newblock On the convergence rate of stochastic approximation for gradient-based stochastic optimization.
\newblock {\em Operations Research}, 73(2):1143--1150.

\bibitem[Hu et~al., 2022]{hu2022stochastic}
Hu, J., Peng, Y., Zhang, G., and Zhang, Q. (2022).
\newblock A stochastic approximation method for simulation-based quantile optimization.
\newblock {\em INFORMS Journal on Computing}, 34(6):2889--2907.

\bibitem[Huang et~al., 2022]{huang2022accelerated}
Huang, F., Gao, S., Pei, J., and Huang, H. (2022).
\newblock Accelerated zeroth-order and first-order momentum methods from mini to minimax optimization.
\newblock {\em Journal of Machine Learning Research}, 23(36):1--70.

\bibitem[Huang et~al., 2023]{huang2023adagda}
Huang, F., Wu, X., and Hu, Z. (2023).
\newblock Adagda: Faster adaptive gradient descent ascent methods for minimax optimization.
\newblock In {\em Proceedings of the 26th International Conference on Artificial Intelligence and Statistics}, volume 206, pages 2365--2389. PMLR.

\bibitem[Kiefer and Wolfowitz, 1952]{kiefer1952stochastic}
Kiefer, J. and Wolfowitz, J. (1952).
\newblock Stochastic estimation of the maximum of a regression function.
\newblock {\em The Annals of Mathematical Statistics}, 23(3):462--466.

\bibitem[Kleijnen et~al., 2010]{kleijnen2010constrained}
Kleijnen, J.~P., Van~Beers, W., and Van~Nieuwenhuyse, I. (2010).
\newblock Constrained optimization in expensive simulation: Novel approach.
\newblock {\em European Journal of Operational Research}, 202(1):164--174.

\bibitem[Kleinman et~al., 1999]{kleinman1999simulation}
Kleinman, N.~L., Spall, J.~C., and Naiman, D.~Q. (1999).
\newblock Simulation-based optimization with stochastic approximation using common random numbers.
\newblock {\em Management Science}, 45(11):1570--1578.

\bibitem[Lam and Zhang, 2025]{lam2025distributionally}
Lam, H. and Zhang, J. (2025).
\newblock Distributionally constrained black-box stochastic gradient estimation and optimization.
\newblock {\em Operations Research}, 73(5):2680--2694.

\bibitem[Li et~al., 2022]{li2022zeroth}
Li, Z., Chen, P.-Y., Liu, S., Lu, S., and Xu, Y. (2022).
\newblock Zeroth-order optimization for composite problems with functional constraints.
\newblock In {\em Proceedings of the AAAI Conference on Artificial Intelligence}, volume~36, pages 7453--7461.

\bibitem[Lin et~al., 2020]{lin2020gradient}
Lin, T., Jin, C., and Jordan, M. (2020).
\newblock On gradient descent ascent for nonconvex-concave minimax problems.
\newblock In {\em Proceedings of the 37th International Conference on Machine Learning}, volume 119, pages 6083--6093. PMLR.

\bibitem[Liu et~al., 2020]{liu2020min}
Liu, S., Lu, S., Chen, X., Feng, Y., Xu, K., Al-Dujaili, A., Hong, M., and O'Reilly, U.-M. (2020).
\newblock Min-max optimization without gradients: Convergence and applications to black-box evasion and poisoning attacks.
\newblock In {\em Proceedings of the 37th International Conference on Machine Learning}, volume 119, pages 6282--6293. PMLR.

\bibitem[Nedi{\'c} and Ozdaglar, 2009]{nedic2009subgradient}
Nedi{\'c}, A. and Ozdaglar, A. (2009).
\newblock Subgradient methods for saddle-point problems.
\newblock {\em Journal of Optimization Theory and Applications}, 142:205--228.

\bibitem[Nesterov and Spokoiny, 2017]{nesterov2017random}
Nesterov, Y. and Spokoiny, V. (2017).
\newblock Random gradient-free minimization of convex functions.
\newblock {\em Foundations of Computational Mathematics}, 17(2):527--566.

\bibitem[Nguyen and Balasubramanian, 2023]{nguyen2023stochastic}
Nguyen, A. and Balasubramanian, K. (2023).
\newblock Stochastic zeroth-order functional constrained optimization: Oracle complexity and applications.
\newblock {\em INFORMS Journal on Optimization}, 5(3):256--272.

\bibitem[Peng et~al., 2016]{peng2016gradient}
Peng, Y., Fu, M.~C., and Hu, J.-Q. (2016).
\newblock Gradient-based simulated maximum likelihood estimation for stochastic volatility models using characteristic functions.
\newblock {\em Quantitative Finance}, 16(9):1393--1411.

\bibitem[Ragueneau et~al., 2024]{ragueneau2024constrained}
Ragueneau, Q., Laurent, L., Legay, A., Larroque, T., and Crambuer, R. (2024).
\newblock A constrained bayesian optimization framework for structural vibrations with local nonlinearities.
\newblock {\em Structural and Multidisciplinary Optimization}, 67(4):47.

\bibitem[Rajwar et~al., 2023]{rajwar2023exhaustive}
Rajwar, K., Deep, K., and Das, S. (2023).
\newblock An exhaustive review of the metaheuristic algorithms for search and optimization: taxonomy, applications, and open challenges.
\newblock {\em Artificial Intelligence Review}, 56(11):13187--13257.

\bibitem[Robbins and Monro, 1951]{robbins1951stochastic}
Robbins, H. and Monro, S. (1951).
\newblock A stochastic approximation method.
\newblock {\em The Annals of Mathematical Statistics}, 22(3):400--407.

\bibitem[Scheinberg, 2022]{scheinberg2022finite}
Scheinberg, K. (2022).
\newblock Finite difference gradient approximation: To randomize or not?
\newblock {\em INFORMS Journal on Computing}, 34(5):2384--2388.

\bibitem[Srinivas et~al., 2010]{srinivas2010gaussian}
Srinivas, N., Krause, A., Kakade, S.~M., and Seeger, M. (2010).
\newblock Gaussian process optimization in the bandit setting: No regret and experimental design.
\newblock In {\em Proceedings of the 27th International Conference on Machine Learning}, pages 1015--1022.

\bibitem[Tsai and Fu, 2014]{tsai2014genetic}
Tsai, S.~C. and Fu, S.~Y. (2014).
\newblock Genetic-algorithm-based simulation optimization considering a single stochastic constraint.
\newblock {\em European Journal of Operational Research}, 236(1):113--125.

\bibitem[Udhayakumar et~al., 2011]{udhayakumar2011stochastic}
Udhayakumar, A., Charles, V., and Kumar, M. (2011).
\newblock Stochastic simulation based genetic algorithm for chance constrained data envelopment analysis problems.
\newblock {\em Omega}, 39(4):387--397.

\bibitem[Ungredda and Branke, 2024]{ungredda2024bayesian}
Ungredda, J. and Branke, J. (2024).
\newblock Bayesian optimisation for constrained problems.
\newblock {\em ACM Transactions on Modeling and Computer Simulation}, 34(2):1--26.

\bibitem[Wang and Spall, 2008]{wang2008stochastic}
Wang, I.-J. and Spall, J.~C. (2008).
\newblock Stochastic optimisation with inequality constraints using simultaneous perturbations and penalty functions.
\newblock {\em International Journal of Control}, 81(8):1232--1238.

\bibitem[Wang et~al., 2025]{wang2025gaussian}
Wang, X., Hong, L.~J., Jiang, Z., and Shen, H. (2025).
\newblock Gaussian process-based random search for continuous optimization via simulation.
\newblock {\em Operations Research}, 73(1):385--407.

\bibitem[Wang et~al., 2023]{wang2023recent}
Wang, X., Jin, Y., Schmitt, S., and Olhofer, M. (2023).
\newblock Recent advances in bayesian optimization.
\newblock {\em ACM Computing Surveys}, 55(13s):1--36.

\bibitem[Xu and Zheng, 2023]{xu2023gradient}
Xu, J. and Zheng, Z. (2023).
\newblock Gradient-based simulation optimization algorithms via multi-resolution system approximations.
\newblock {\em INFORMS Journal on Computing}, 35(3):633--651.

\bibitem[Xu et~al., 2023]{xu2023unified}
Xu, Z., Zhang, H., Xu, Y., and Lan, G. (2023).
\newblock A unified single-loop alternating gradient projection algorithm for nonconvex--concave and convex--nonconcave minimax problems.
\newblock {\em Mathematical Programming}, 201(1):635--706.

\bibitem[Zhou and Bhatnagar, 2017]{zhou2017gradient}
Zhou, E. and Bhatnagar, S. (2017).
\newblock Gradient-based adaptive stochastic search for simulation optimization over continuous space.
\newblock {\em INFORMS Journal on Computing}, 30(1):154--167.

\end{thebibliography}

\newpage

\begin{center}
    {\LARGE\bfseries Appendices}
\end{center}

\appendix
\section{Proofs of Lemmas for Theorem \ref{thm:Opt}} \label{app:proofs_thm_opt}
\subsection{Proof of Lemma \ref{lemma:measure_bound}}
Denote $G_x(x_t,y_t;g)=\frac{1}{\gamma}(x_t-\mathcal{P}_\mathcal{X}[x_t-\gamma g])$ for some $g\in \mathbb{R}^{d_x}$. Then, we can decompose $\mathfrak{g}_x(x_t,y_t)$ by
\begin{align*}
\mathfrak{g}_x(x_t,y_t)=&G_x(x_t,y_t;\nabla_x h(x_t,y_t))-G_x(x_t,y_t;\nabla_x h_1 (x_t,y_t))\\
&+G_x(x_t,y_t;\nabla_x h_1 (x_t,y_t))-G_x(x_t,y_t;v_t)+G_x(x_t,y_t;v_t).
\end{align*}
By the non-expansiveness of projection, we bound $\|\mathfrak{g}_x(x_t,y_t)\|^2$ by
\begin{align*}
\|\mathfrak{g}_x(x_t,y_t)\|^2 & =\left\|\frac{1}{\gamma}\left(x_t-\mathcal{P}_\mathcal{X}[x_t-\gamma\nabla_xh(x_t,y_t)]\right)\right\|^2\\
% &\leq \!\left\|\frac{1}{\gamma}\left(x_t-\mathcal{P}_\mathcal{X}[x_t-\gamma v_t]\right)+v_t-\nabla_x f(x_t,y_t)\right\|^2\\
&\leq 3\|\nabla_x f\!(x_t,y_t)\!-\!\nabla_x f_1\!(x_t,y_t)\|^2\!+\!3\!\left\| \Delta_v^t \right\|^2 \!\!+\! \frac{3}{\gamma^2}\!\left\| \Delta_x^t \right\|^2\\
&\leq \frac{3}{\gamma^2}\left\| \Delta_x^t \right\|^2+3\left\| \Delta_v^t \right\|^2+3L^2r^2.
\end{align*}
where the second step follows from Cauchy-Schwarz inequality and the last step follows from Lemma \ref{lemma:esti_bias}. Similarly,
denote $G_y(x_t,y_t;g)=\frac{1}{\lambda}(y_t-\mathcal{P}_\mathcal{Y}[y_t+\lambda g])$ for some $g\in \mathbb{R}^{d_y}$ and decompose $\mathfrak{g}_y(x_t,y_t)$ by
\begin{align*}
\mathfrak{g}_y(x_t,y_t)=&G_y(x_t,y_t;\nabla_y h(x_t,y_t))-G_y(x_t,y_t;\nabla_y h_2(x_t,y_t))+G_y(x_t,y_t;\nabla_y h_2(x_t,y_t))\\
& -G_y(x_t,y_t;\nabla_y f_2(x_t,y_t))+G_y(x_t,y_t;\nabla_y f_2(x_t,y_t))-G_y(x_t,y_t;w_t)+G_y(x_t,y_t;w_t).
\end{align*}
Using Lemma \ref{lemma:esti_bias}, the non-expansiveness of projection, and Cauchy-Schwarz inequality, we can bound $\|\mathfrak{g}_y(x_t,y_t)\|^2$ by
$$\|\mathfrak{g}_y(x_t,y_t)\|^2=\big\|\frac{1}{\lambda}\left(y_t-\mathcal{P}_\mathcal{Y}[y_t+\lambda\nabla_y h(x_t,y_t)]\right)\big\|^2\leq \frac{4}{\lambda^2}\left\| \Delta_y^t \right\|^2+4\left\| \Delta_w^t \right\|^2+4L^2r^2+4\mu^2D_y^2.$$
% \begin{align*}
% &\|\mathfrak{g}_y(x_t,y_t)\|^2=\big\|\frac{1}{\lambda}\left(y_t-\mathcal{P}_\mathcal{Y}[y_t+\lambda\nabla_y h(x_t,y_t)]\right)\big\|^2\\
% &\leq \left\|\frac{1}{\lambda}\left(y_t\!-\!\mathcal{P}_\mathcal{Y}[y_t+\lambda w_t]\right)\!+\!\nabla_y f(x_t,y_t)\!+\!\mu y_t\!-\!w_t\right\|^2\\
% &\leq \frac{4}{\lambda^2}\left\| \Delta_y^t \right\|^2+4\left\| \Delta_w^t \right\|^2+4L^2r^2d_y+4\mu^2D_y^2.
% \end{align*}
% where $D_y=\max_{y_1,y_2\in\mathcal{Y}} \|y_1-y_2\|$.
Combining the above two inequalities finishes the proof.

\subsection{Proof of Lemma \ref{lemma:F_diff_bound}}
Denote $F(x)=\max_{y\in\mathcal{Y}}f(x,y)$. According to Lemma 4.3 of \cite{lin2020gradient}, $F(x)$ and $F_r(x)$ both have $L_F$-Lipschitz gradients, and $y_r^*(x)$ is $\frac{L}{\mu}$-Lipschitz continuous.
Using the smoothness of $F_r$, we have
\begin{align}\label{eq:proof_F_diff_1}
\Delta_F^t \leq\; & \langle \nabla F_r(x_t),x_{t+1}-x_t\rangle+\frac{L_F}{2}\|x_{t+1}-x_t\|^2\nonumber\\
=\; & \eta_t \langle \nabla F_r(x_t)-v_t,\Delta_x^t\rangle +\eta_t \langle  v_t,\Delta_x^t\rangle+\frac{L_F\eta_t^2}{2}\|\Delta_x^t\|^2,
\end{align}
where we used the relationship $x_{t+1}-x_t=\eta_t\Delta_x^t$ in the second step. Based on the projection step: $\Tilde{x}_{t+1}=\mathcal{P}_\mathcal{X}[x_t-\gamma v_t]$, we have $\langle\Tilde{x}_{t+1}-x_t+\gamma v_t, x-\Tilde{x}_{t+1}\rangle \geq 0$
% \begin{align*}
% \langle\Tilde{x}_{t+1}-x_t+\gamma v_t, x-\Tilde{x}_{t+1}\rangle \geq 0,
% \end{align*}
for any $x\in\mathcal{X}$. Let $x=x_t$, we can get
\begin{align}\label{eq:proof_F_diff_2}
\langle\Tilde{x}_{t+1}-x_t, v_t\rangle \leq -\frac{1}{\gamma}\| \Tilde{x}_{t+1}-x_t \|^2.
\end{align}
Then, we decompose $\langle \nabla F_r(x_t)-v_t,\Delta_x^t\rangle$ into two terms:
\begin{align*}
\langle \nabla F_r(x_t)-v_t,\Delta_x^t\rangle
=& \langle \nabla F_r(x_t) -\nabla_x f_{1}(x_t,y_t),\Delta_x^t\rangle+ \langle \nabla_x f_{1}(x_t,y_t)-v_t,\Delta_x^t\rangle=G_1+G_2,
\end{align*}
where we denote $G_1=\langle \nabla F_r(x_t) -\nabla_x f_{1}(x_t,y_t),\Delta_x^t\rangle$ and $G_2=\langle \nabla_x f_{1}(x_t,y_t)-v_t,\Delta_x^t\rangle$.
It is not hard to verify that $f_1(x,y;\xi)$ also has $ L$-Lipschitz gradients. Then, we bound $G_1$ as follows.
\begin{align*}
G_1 & \leq \|\nabla F_r(x_t)-\nabla_x f_{1}(x_t,y_t)\|\|\Delta_x^t\|\\
% &\leq 2\gamma\|\nabla F_r(x_t)-\nabla_x f_{1}(x_t,y_t)\|^2+\frac{1}{8\gamma}\|\Delta_x^t\|^2\\
& =2\gamma\|\nabla_x f_1(x_t,y_r^*(x_t))\!-\!\nabla_x f_{1}(x_t,y_t)\|^2\!+\!\frac{\|\Delta_x^t\|^2}{8\gamma}\\
& \leq 2L^2\gamma\|y_r^*(x_t)-y_t\|^2+ \frac{1}{8\gamma}\|\Delta_x^t\|^2,
\end{align*}
where we used the Cauchy-Schwarz inequality and $\nabla F_r(x_t)=\nabla_x f_1(x_t,y_r^*(x_t))$ in the second step. As for $G_2$, we have $G_2\leq 2\gamma\|\Delta_v^t\|^2+\frac{1}{8\gamma}\|\Delta_x^t\|^2$.
% \begin{align*}
% G_2\leq 2\gamma\|\Delta_v^t\|^2+\frac{1}{8\gamma}\|\Delta_x^t\|^2.
% \end{align*}
Combining the bounds for $G_1$ and $G_2$ leads to
\begin{align}\label{eq:proof_F_diff_3}
\langle \nabla F_r(x_t)-v_t,\Delta_x^t\rangle
\leq  2L^2\gamma\|y_r^*(x_t)-y_t\|^2+2\gamma\|\Delta_v^t\|^2+\frac{1}{4\gamma}\|\Delta_x^t\|^2.
\end{align}
Substituting \eqref{eq:proof_F_diff_2} and \eqref{eq:proof_F_diff_3} into \eqref{eq:proof_F_diff_1}, we can get
\begin{align*}
\Delta_F^t \leq & 2L^2\gamma \eta_t\|y_r^*(x_t)-y_t\|^2+ 2\gamma\eta_t\|\Delta_v^t\|^2 +\left(\frac{\eta_t}{4\gamma}+\frac{L_F\eta_t^2}{2}-\frac{\eta_t}{\gamma}\right)\|\Delta_x^t\|^2\\
\leq &  2L^2\gamma \eta_t\|y_r^*(x_t)\!-\!y_t\|^2\!+\!2\gamma\eta_t\|\Delta_v^t\|^2\!-\!\frac{\eta_t}{2\gamma}\|\Delta_x^t\|^2,
\end{align*}
where we used the condition $k_1\leq \frac{1}{2L_F\eta_0\mu^{3-e}}$ (which leads to $\gamma\leq \frac{1}{2L_F\eta_0}\leq \frac{1}{2L_F\eta_t}$) in the last step.

\subsection{Proof of Lemma \ref{lemma:y_diff_bound}}
We have $f_1(x,y)$ is $\mu$-strongly concave in $y$ because it is the expectation of the $\mu$-strongly concave function $f(x+rz,\cdot)$. Therefore, for any $y\in\mathcal{Y}$, 
\begin{align}\label{eq:proof_y_diff_1}
f_1(x_t,y)-f_1(x_t,y_t)
\leq\; \langle \nabla_y f_1(x_t,y_t), y-y_t\rangle-\frac{\mu}{2}\|y-y_t\|^2.
% \nonumber\\
% =\; &\langle w_t, y-\Tilde{y}_{t+1}\rangle +\langle \nabla_y f_1(x_t,y_t)-w_t, y-\Tilde{y}_{t+1}\rangle + \langle \nabla_y f_1(x_t,y_t), \Tilde{y}_{t+1}-y_t\rangle -\frac{\mu}{2}\|y-y_t\|^2.
\end{align}
Using the smoothness of $f_1(x,y)$, we have
\begin{align}\label{eq:proof_y_diff_2}
f_1(x_t,\Tilde{y}_{t+1})-f_1(x_t,y_t)\geq \langle \nabla_y f_1(x_t,y_t), \Delta_y^t\rangle-\frac{L}{2}\|\Delta_y^t\|^2\!\!.
\end{align}
% \begin{align}\label{eq:proof_y_diff_2}
% &f(x_t,\Tilde{y}_{t+1})-f(x_t,y_t)\\
% \geq &\langle \nabla_y f(x_t,y_t), \Tilde{y}_{t+1}-y_t\rangle-\frac{L}{2}\|\Tilde{y}_{t+1}-y_t\|^2.
% \end{align}
Denote $\Tilde{\Delta}_w^t=\nabla_y f_1(x_t,y_t)-w_t$. Combining \eqref{eq:proof_y_diff_1} by setting $y=y_r^*(x_t)$ with \eqref{eq:proof_y_diff_2} leads to
\begin{align}\label{eq:proof_y_diff_3}
 0\leq & f_1(x_t,y_r^*(x_t))-f_1(x_t,\Tilde{y}_{t+1})\nonumber\\
\leq\;& \underbrace{\langle 
w_t, y_r^*(x_t)-\Tilde{y}_{t+1} \rangle}_{C_1}+\underbrace{\langle 
\Tilde{\Delta}_w^t, y_r^*(x_t)-\Tilde{y}_{t+1} \rangle}_{C_2} + \frac{L}{2}\| \Tilde{y}_{t+1}-y_t \|^2 -\frac{\mu}{2}\| y_r^*(x_t)-y_t \|^2.
\end{align}
Then, we try to bound $C_1$ and $C_2$, respectively. Using the property of projection for the dual step, we have $\left\langle -w_t+\frac{1}{\lambda}(\Tilde{y}_{t+1}-y_t), y-\Tilde{y}_{t+1} \right\rangle \geq 0$, which leads to $$C_1 \leq \big\langle \frac{1}{\lambda}(\Tilde{y}_{t+1}-y_t), y_r^*(x_t)-\Tilde{y}_{t+1} \big\rangle 
= -\frac{1}{\lambda} \|\Delta_y^t\|^2 +\frac{1}{\lambda}\langle \Delta_y^t, y_r^*(x_t)-y_t\rangle.$$
% \begin{align*}
% \left\langle -w_t+\frac{1}{\lambda}(\Tilde{y}_{t+1}-y_t), y-\Tilde{y}_{t+1} \right\rangle \geq 0.
% \end{align*}
% Denote $\|x\|_A=\sqrt{x^TAx}$ for some vector $x\in\mathbb{R}^{d_x}$ and $A\in\mathbb{R}^{d_x\times d_x}$. 
% Then, the above inequality \eqref{eq:proof_F_diff_3} leads to
% \begin{align*}
% C_1& \leq \left\langle \frac{1}{\lambda}(\Tilde{y}_{t+1}-y_t), y_r^*(x_t)-\Tilde{y}_{t+1} \right\rangle \\
% &= -\frac{1}{\lambda} \|\Delta_y^t\|^2 +\frac{1}{\lambda}\langle \Delta_y^t, y_r^*(x_t)-y_t\rangle.
% \end{align*}
We also have the equation: 
$$\|y_{t+1}-y_r^*(x_t)\|^2 =\|y_{t}+\eta_t(\Tilde{y}_{t+1}-y_t)-y_r^*(x_t)\|^2 =\|y_{t}-y_r^*(x_t)\|^2+2\eta_t\langle \Delta_y^t, y_t-y_r^*(x_t) \rangle+\eta_t^2\|\Delta_y^t\|^2.$$
% \begin{align*}
% &\|y_{t+1}-y_r^*(x_t)\|^2 =\|y_{t}+\eta_t(\Tilde{y}_{t+1}-y_t)-y_r^*(x_t)\|^2\\
% & =\|y_{t}-y_r^*(x_t)\|^2+2\eta_t\langle \Delta_y^t, y_t-y_r^*(x_t) \rangle+\eta_t^2\|\Delta_y^t\|^2.
% \end{align*}
Using the above equality and the fact $\eta_t\leq 1$, we further have the bound for $C_1$:
$$C_1\leq -\frac{1}{2\lambda} \|\Delta_y^t\|^2+\frac{1}{2\lambda \eta_t} \|y_{t}-y_r^*(x_t)\|^2-\frac{1}{2\lambda \eta_t} \|y_{t+1}-y_r^*(x_t)\|^2.$$
% \begin{align*}
% C_1\leq &-\frac{1}{2\lambda} \|\Delta_y^t\|^2+\frac{1}{2\lambda \eta_t} \|y_{t}-y_r^*(x_t)\|^2\\
% &-\frac{1}{2\lambda \eta_t} \|y_{t+1}-y_r^*(x_t)\|^2.
% \end{align*}
Next, we bound $C_2$. By Young’s inequality, 
$$C_2 = \langle \Tilde{\Delta}_w^t, y_r^*(x_t)- y_t \rangle+\langle 
\Tilde{\Delta}_w^t, y_t-\Tilde{y}_{t+1} \rangle \leq \frac{\mu}{4}\|y_r^*(x_t)- y_t\|^2+\frac{2}{\mu}\|\Tilde{\Delta}_w^t\|^2+\frac{\mu}{4}\|\Delta_y^t\|^2.$$
% \begin{align*}
% C_2&= \langle \Tilde{\Delta}_w^t, y_r^*(x_t)- y_t \rangle+\langle 
% \Tilde{\Delta}_w^t, y_t-\Tilde{y}_{t+1} \rangle\\
% & \leq \frac{\mu}{4}\|y_r^*(x_t)- y_t\|^2+\frac{2}{\mu}\|\Tilde{\Delta}_w^t\|^2+\frac{\mu}{4}\|\Delta_y^t\|^2.
% \end{align*}
Substituting the bounds for $C_1$ and $C_2$ into \eqref{eq:proof_y_diff_3}, we can get
\begin{equation}
\begin{aligned}\label{eq:bound_yy}
\|y_{t+1}-y_r^*(x_t)\|^2 \leq\; & \left(1-\frac{\lambda\mu\eta_t}{2}\right)\|y_r^*(x_t)- y_t\|^2+\frac{4\lambda\eta_t}{\mu} \|\Tilde{\Delta}_w^t\|^2+\lambda\eta_t\left(\frac{\mu}{2}+L-\frac{1}{\lambda}\right)\|\Delta_y^t\|^2\\
 \leq &\left( 1-\frac{\lambda\mu\eta_t}{2} \right) \|\Delta_*^t\|^2-\frac{3\eta_t}{4}\|\Delta_y^t\|^2 +\frac{4\lambda\eta_t}{\mu}\|\Tilde{\Delta}_w^t\|^2, 
\end{aligned}
\end{equation}
where the last step follows from the relations $k_2\leq \frac{1}{6L\mu^{1-e}}$ (which leads to $\lambda \leq \frac{1}{6L}$), and $\mu\leq \frac{2L}{3}$. 
% we can rearrange the above inequality as
% \begin{align*}
% \|y_{t+1}-y_r^*(x_t)\|^2 \leq &\left( 1-\frac{\lambda\mu\eta_t}{2} \right) \|y_r^*(x_t)- y_t\|^2\\
% &-\frac{3\eta_t}{4}\|\Delta_y^t\|^2 +\frac{4\lambda\eta_t}{\mu}\|\Tilde{\Delta}_w^t\|^2.
% \end{align*}
Using the inequality: 
$$2\langle y_{t+1}-y_r^*(x_t), y_r^*(x_t)-y_r^*(x_{t+1}) \rangle\leq \frac{\lambda\mu\eta_t}{4}\|y_{t+1}-y_r^*(x_{t})\|^2+\frac{4}{\lambda\mu\eta_t}\|y_r^*(x_{t})-y_r^*(x_{t+1})\|^2,$$
we can further derive the following bound:
$$\|y_{t+1}-y_r^*(x_{t+1})\|^2  
=\; \|y_{t+1}-y_r^*(x_{t})+y_r^*(x_{t})-y_r^*(x_{t+1})\|^2
\leq\; D_1+D_2,$$
% \begin{align*}
% &\|y_{t+1}-y_r^*(x_{t+1})\|^2  \\
% =\;& \|y_{t+1}-y_r^*(x_{t})+y_r^*(x_{t})-y_r^*(x_{t+1})\|^2\\
% \leq\;& D_1+D_2,
% \end{align*}
where we denote $D_1=\left( 1+\frac{\lambda\mu\eta_t}{4} \right)\|y_{t+1}-y_r^*(x_{t})\|^2$, and $D_2=\left( 1+\frac{4}{\lambda\mu\eta_t} \right)\|y_r^*(x_{t})-y_r^*(x_{t+1})\|^2$.
Using the bound of $\|y_{t+1}-y_r^*(x_t)\|^2$ in \eqref{eq:bound_yy} and the fact $\lambda\mu\eta_t\leq 1$, we can control $D_1$ by:
\begin{align*}
D_1\leq & \left( 1-\frac{\lambda\mu\eta_t}{4} \right)\|\Delta_*^t\|^2-\frac{3\eta_t}{4}\|\Delta_y^t\|^2+\frac{6\lambda\eta_t}{\mu}\|\Tilde{\Delta}_w^t\|^2\\
\leq & \left( 1-\frac{\lambda\mu\eta_t}{4} \right)\|\Delta_*^t\|^2-\frac{3\eta_t}{4}\|\Delta_y^t\|^2 +\frac{12\lambda\eta_t}{\mu}\|\Delta_w^t\|^2+\frac{24L^2r^2\lambda \eta_t}{\mu}.
\end{align*}
In the last step, we applied the following inequality:
\begin{align*}
\|\nabla_y f_1(x_t,y_t)-w_t\|^2&=\|\nabla_y f_1(x_t,y_t)-\nabla_y f_2(x_t,y_t)+\nabla_y f_2(x_t,y_t)-w_t\|^2\\
&\leq 2\|\Delta_w^t\|^2+2\|\nabla_y f_1(x_t,y_t)-\nabla_y f(x_t,y_t)+\nabla_y f(x_t,y_t)-\nabla_y f_2(x_t,y_t)\|^2\\
&\leq 2\|\Delta_w^t\|^2+ 4L^2r^2.
\end{align*}
As for $D_2$, we use the smoothness of $y_r^*$ to derive 
$$D_2\leq  \left( 1+\frac{4}{\lambda\mu\eta_t} \right)\kappa^2\|x_{t+1}-x_{t}\|^2
\leq \frac{5\kappa^2\eta_t}{\lambda\mu} \|\Delta_x^t\|^2,$$
% \begin{align*}
% D_2\leq & \left( 1+\frac{4}{\lambda\mu\eta_t} \right)\kappa^2\|x_{t+1}-x_{t}\|^2\\
% \leq & \frac{5\kappa^2\eta_t}{\lambda\mu} \|\Delta_x^t\|^2,
% \end{align*}
where we also used the fact $\mu\lambda\eta_t\leq 1$. Combining the bounds for $D_1$ and $D_2$ finished the proof.

% \begin{align*}
% &\|y_{t+1}-y_r^*(x_{t+1})\|^2  \\
% =\;& \|y_{t+1}-y_r^*(x_{t})+y_r^*(x_{t})-y_r^*(x_{t+1})\|^2\\
% \leq&\left( 1+\frac{\lambda\mu\eta_t}{4} \right)\|y_{t+1}-y_r^*(x_{t})\|^2+ \left( 1+\frac{4}{\lambda\mu\eta_t} \right)\|y_r^*(x_{t})-y_r^*(x_{t+1})\|^2\\
% \leq\;& \left( 1+\frac{\lambda\mu\eta_t}{4} \right)\|y_{t+1}-y_r^*(x_{t})\|^2+ \left( 1+\frac{4}{\lambda\mu\eta_t} \right)\kappa^2\|x_{t+1}-x_{t}\|^2\\
% \leq\;& \left( 1-\frac{\lambda\mu\eta_t}{2} \right)\|y_{t}-y_r^*(x_{t})\|^2-\frac{3\eta_t}{4}\|\Tilde{y}_{t+1}-y_t\|^2+\frac{6\lambda\eta_t}{\mu}\|\nabla_y f(x_t,y_t)-w_t\|^2+\frac{5\kappa^2\eta_t}{\lambda\mu} \|\Tilde{x}_{t+1}-x_{t}\|^2\\
% \leq\; & \left( 1-\frac{\lambda\mu\eta_t}{2} \right)\|y_{t}-y_r^*(x_{t})\|^2-\frac{3\eta_t}{4}\|\Tilde{y}_{t+1}-y_t\|^2+\frac{12\lambda\eta_t}{\mu}\|\nabla_y f_{r_2}(x_t,y_t)-w_t\|^2\\
% & +\frac{5\kappa^2\eta_t}{\lambda\mu} \|\Tilde{x}_{t+1}-x_{t}\|^2+\frac{3L^2d_yr_2^2\lambda \eta_t}{\mu},
% \end{align*}
% where we used the smoothness of $y_r^*(x)$ in the second inequality, and $\frac{\mu\lambda\eta_t}{2\Bar{b}}\leq \frac{1}{2}$ in the third inequality. 

\subsection{Proof of Lemma \ref{lemma:v_error_bound}}
Based on the update for $v_t$, we have
% $$v_{t+1}-v_t= -\alpha_{t+1}v_t+\alpha_{t+1}\hat{\nabla}_x f(x_{t+1},y_{t+1};\mathcal{S}_{t+1})
% \!+\!(1\!-\!\alpha_{t+1})(\hat{\nabla}_xf\!(x_{t+1},y_{t+1};\mathcal{S}_{t+1}\!)\!-\!\hat{\nabla}_x f\!(x_t,y_t;\mathcal{S}_{t+1})).$$
$$
v_{t+1}-v_t= -\alpha_{t+1}v_t+\alpha_{t+1}\hat{\nabla}_x f(x_{t+1},y_{t+1};\mathcal{S}_{t+1})+(1-\alpha_{t+1})(\hat{\nabla}_xf(x_{t+1},y_{t+1};\mathcal{S}_{t+1})-\hat{\nabla}_x f(x_t,y_t;\mathcal{S}_{t+1})).
$$
Then, we can bound $\mathbb{E}\left[\|\Delta_v^{t+1}\|^2\right]$ by
\begin{align*}
&\mathbb{E}\left[\|\Delta_v^{t+1}\|^2\right]\\
=\;&\mathbb{E}\left[\|\nabla_x f_{1}(x_{t+1},y_{t+1})-v_t-(v_{t+1}-v_t)\|^2\right]\\
=\;& \mathbb{E}\left[\left\|\nabla_x f_{1}(x_{t+1},y_{t+1})-v_{t}+\alpha_{t+1}v_t\right.\right.-(1-\alpha_{t+1})\left(\hat{\nabla}_x f(x_{t+1},y_{t+1};\mathcal{S}_{t+1})-\hat{\nabla}_x f(x_t,y_t;\mathcal{S}_{t+1})\right)\\
&\qquad-\alpha_{t+1}\hat{\nabla}_x f(x_{t+1},y_{t+1};\mathcal{S}_{t+1})\left.\left.\right\|^2\right].
% \\
% \leq\;&\mathbb{E}\left[\left\|(1-\alpha_{t+1})\underbrace{\left(\nabla_x f_{r_1}(x_{t},y_{t})-v_{t}\right)}_{a_1}+\alpha_{t+1}\underbrace{\left( \nabla_x f_{r_1}(x_{t+1},y_{t+1})- \hat{\nabla}_xf(x_{t+1},y_{t+1};\mathcal{B}_{t+1})\right)}_{a_2}\right.\right.\\
% &+ (1-\alpha_{t+1})\left(\nabla_x f_{r_1}(x_{t+1},y_{t+1})-\left(\nabla_x f_{r_1}(x_{t},y_{t})-v_{t}\right) -\hat{\nabla}_x f(x_{t+1},y_{t+1};\mathcal{B}_{t+1})+\hat{\nabla}_x f(x_{t},y_{t};\mathcal{B}_{t+1})\right)\left.\left.\left.\right\|^2\right| \mathcal{F}_{t}\right]
\end{align*}
For notational simplicity, we denote 
% $a_1=\nabla_x f_{1}(x_{t},y_{t})-v_{t}$, $a_2=\nabla_x f_{1}(x_{t+1},y_{t+1})-\nabla_x f_{1}(x_{t},y_{t}) -\hat{\nabla}_x f(x_{t+1},y_{t+1};\mathcal{S}_{t+1})+\hat{\nabla}_x f(x_{t},y_{t};\mathcal{S}_{t+1})$, and $a_3=\nabla_x f_{1}(x_{t+1},y_{t+1})- \hat{\nabla}_xf(x_{t+1},y_{t+1};\mathcal{S}_{t+1})$.
\begin{align*}
a_1&=\nabla_x f_{1}(x_{t},y_{t})-v_{t},\\
a_2&=\nabla_x f_{1}(x_{t+1},y_{t+1})-\nabla_x f_{1}(x_{t},y_{t}) -\hat{\nabla}_x f(x_{t+1},y_{t+1};\mathcal{S}_{t+1})+\hat{\nabla}_x f(x_{t},y_{t};\mathcal{S}_{t+1}),\\
a_3&=\nabla_x f_{1}(x_{t+1},y_{t+1})- \hat{\nabla}_xf(x_{t+1},y_{t+1};\mathcal{S}_{t+1}).
\end{align*}
Then we can further decompose $\mathbb{E}\left[\|\Delta_v^{t+1}\|^2\right]$ by
\begin{align}\label{eq:proof_v_diff_1}
&\mathbb{E}\left[\|\Delta_v^{t+1}\|^2\right]\nonumber\\
=\;&\mathbb{E}\left[\left\| (1-\alpha_{t+1})a_1+(1-\alpha_{t+1})a_2+\alpha_{t+1}a_3
 \right\|^2\right]\nonumber\\
 =\;& \mathbb{E}\left[(1-\alpha_{t+1})^2 \left\| a_1\right\|^2+(1-\alpha_{t+1})^2\left\| a_2\right\|^2+ 2(1-\alpha_{t+1})^2a_1^{\text{T}}a_2+2\alpha_{t+1}(1-\alpha_{t+1})a_2^{\text{T}}a_3\right]\nonumber\\
 &+\mathbb{E}\left[\alpha_{t+1}^2\left\| a_3\right\|^2+2\alpha_{t+1}(1-\alpha_{t+1})a_1^{\text{T}}a_3\right].
\end{align}
Using the relations $\mathbb{E}\left[ \hat{\nabla}_xf(x_{t},y_{t};\mathcal{S}_{t+1}) \right]=\nabla_x f_{1}(x_{t},y_{t})$, and $\mathbb{E}\left[ \hat{\nabla}_xf(x_{t+1},y_{t+1};\mathcal{S}_{t+1}) \right]=\nabla_x f_{1}(x_{t+1},y_{t+1})$,
% \begin{align*}
%     &\mathbb{E}\left[ \hat{\nabla}_xf(x_{t},y_{t};\mathcal{S}_{t+1}) \right]=\nabla_x f_{1}(x_{t},y_{t}),\\
%     &\mathbb{E}\left[ \hat{\nabla}_xf(x_{t+1},y_{t+1};\mathcal{S}_{t+1}) \right]=\nabla_x f_{1}(x_{t+1},y_{t+1}),
% \end{align*}
we can get $\mathbb{E}\left[ a_1^{\text{T}}a_2]=\mathbb{E}[a_1^{\text{T}}a_3\right]=0$.
Thus, we can arrange \eqref{eq:proof_v_diff_1} to get
\begin{align*}
& \mathbb{E}\left[\|\Delta_v^{t+1}\|^2\right]\\
=\;& \mathbb{E}\left[(1-\alpha_{t+1})^2 \left\| a_1\right\|^2+(1-\alpha_{t+1})^2\left\| a_2\right\|^2+\alpha_{t+1}^2\left\| a_3\right\|^2+2\alpha_{t+1}(1-\alpha_{t+1})a_2^{\text{T}}a_3\right]\\
\leq\;&\mathbb{E}\left[(1-\alpha_{t+1})^2 (\left\| a_1\right\|^2+2\left\| a_2\right\|^2)+2\alpha_{t+1}^2\left\| a_3\right\|^2\right]\\
\leq\;& \mathbb{E}\left[(1-\alpha_{t+1})^2 (\left\| a_1\right\|^2+2\left\| a_2\right\|^2)\right]+\frac{2\alpha_{t+1}^2\sigma^2}{q}.
\end{align*}
Since $\mathbb{E}[\|\varsigma-\mathbb{E}[\varsigma]\|^2]=\mathbb{E}[\|\varsigma\|^2]-\|\mathbb{E}[\varsigma]\|^2$ for any random variable $\varsigma$, we have for any $\zeta_{t+1}^i$
\begin{align}
& \mathbb{E}[\|a_2\|^2]\notag \\
\leq\;& \mathbb{E}\left[\left\|\hat{\nabla}_x f(x_{t+1},y_{t+1};\mathcal{S}_{t+1})-\hat{\nabla}_x f(x_{t},y_{t};\mathcal{S}_{t+1})\right\|^2\right]\notag \\
\leq\;& \frac{1}{q}\mathbb{E}\left[\left\|\hat{\nabla}_x f(x_{t+1},y_{t+1};\zeta_{t+1}^i)\!-\!\hat{\nabla}_x f(x_{t},y_{t};\zeta_{t+1}^i)\right\|^2\right]\notag\\
=\;& \frac{d_x^2}{q}\mathbb{E}\left[\left\|\frac{f(x_{t+1}+rz_{t+1}^i,y_{t+1};\xi_{t+1}^i)-f(x_{t+1},y_{t+1};\xi_{t+1}^i)}{r}z_{t+1}^i-\frac{f(x_{t}+rz_{t+1}^i,y_{t};\xi_{t+1}^i)-f(x_{t},y_{t};\xi_{t+1}^i)}{r}z_{t+1}^i\right\|^2\right]\notag\\
=\; & \frac{d_x^2}{q}\mathbb{E}\left[\left\| \frac{f(x_{t+1}+rz_{t+1}^i,y_{t+1};\xi_{t+1}^i)-f(x_{t+1},y_{t+1};\xi_{t+1}^i)-\langle \nabla_x f(x_{t+1},y_{t+1};\zeta_{t+1}^i), r z_{t+1}^i\rangle}{r}z_{t+1}^i \right.\right.\notag\\
& + \langle\nabla_x f(x_{t},y_{t};\zeta_{t+1}^i)-\nabla_x f(x_{t+1},y_{t+1};\zeta_{t+1}^i), z_{t+1}^i\rangle z_{t+1}^i \notag\\
& \left.\left. - \frac{f(x_{t}+rz_{t+1}^i,y_{t};\xi_{t+1}^i)-f(x_{t},y_{t};\xi_{t+1}^i)-\langle\nabla_x f(x_{t},y_{t};\zeta_{t+1}^i), r z_{t+1}^i\rangle}{r}z_{t+1}^i \right\|^2 \right].\label{eq:a2_mid}
% \leq & \frac{3}{q}\mathbb{E}\left[\left\| \langle \nabla_x f(x_{t+1},y_{t+1};\zeta_{t+1}^i)-\nabla_x f(x_{t},y_{t};\zeta_{t+1}^i), z_{t+1}^i\rangle z_{t+1}^i\right\|^2\right]+\frac{45L^2r^2d_x}{2q},
\end{align}
Using the smoothness of $f$ in $x$, we can get
\begin{align*}
f(x_{t+1}+rz_{t+1}^i,y_{t+1};\xi_{t+1}^i)-f(x_{t+1},y_{t+1};\xi_{t+1}^i)-\langle \nabla_x f(x_{t+1},y_{t+1};\zeta_{t+1}^i), r z_{t+1}^i\rangle & \leq \frac{Lr^2}{2}\|z_{t+1}^i\|^2,\\
f(x_{t}+rz_{t+1}^i,y_{t};\xi_{t+1}^i)-f(x_{t},y_{t};\xi_{t+1}^i)-\langle\nabla_x f(x_{t},y_{t};\zeta_{t+1}^i), r z_{t+1}^i\rangle &\leq \frac{Lr^2}{2}\|z_{t+1}^i\|^2.
\end{align*}
Combining the above two inequalities with \eqref{eq:a2_mid} and using the Cauchy-Schwarz inequality, we have
\begin{align*}
    \mathbb{E}[\|a_2\|^2]\leq & \frac{3d_x^2}{q}\mathbb{E}\left[\left\| \langle \nabla_x f(x_{t+1},y_{t+1};\zeta_{t+1}^i)-\nabla_x f(x_{t},y_{t};\zeta_{t+1}^i), z_{t+1}^i\rangle z_{t+1}^i\right\|^2\right]+\frac{45L^2r^2d_x^2}{2q},
\end{align*}
where we also used the fact that $\mathbb{E}[\|z_{t+1}^i\|^6]=\frac{(d_x+2)(d_x+4)}{d_x^2}\leq 15$. Denote $G=\mathbb{E}_{z_{t+1}^i}\left[\|z_{t+1}^i\|^2z_{t+1}^i(z_{t+1}^i)^T\right]$. We have $G=\frac{d_x+2}{d_x^2}I_{d_x}$ \citep{berahas2022theoretical} and
% $\mathbb{E}\big[\left\| \langle \nabla_x f(x_{t+1},y_{t+1};\zeta_{t+1}^i)-\nabla_x f(x_{t},y_{t};\zeta_{t+1}^i),z_{t+1}^i\rangle z_{t+1}^i\right\|^2\big]=\mathbb{E}_{\xi_{t+1}^i}\big[\big(\nabla_x f(x_{t+1},y_{t+1};\zeta_{t+1}^i)-\nabla_x f(x_{t},y_{t};\zeta_{t+1}^i)\big)^T G\big(\nabla_x f(x_{t+1},y_{t+1};\zeta_{t+1}^i)-\nabla_x f(x_{t},y_{t};\zeta_{t+1}^i)\big)\big]$.
\begin{align*}
& \mathbb{E}\big[\left\| \langle \nabla_x f(x_{t+1},y_{t+1};\zeta_{t+1}^i)-\nabla_x f(x_{t},y_{t};\zeta_{t+1}^i),z_{t+1}^i\rangle z_{t+1}^i\right\|^2\big]\\
=\;&\mathbb{E}_{\xi_{t+1}^i}\big[\big(\nabla_x f(x_{t+1},y_{t+1};\zeta_{t+1}^i)-\nabla_x f(x_{t},y_{t};\zeta_{t+1}^i)\big)^T G\big(\nabla_x f(x_{t+1},y_{t+1};\zeta_{t+1}^i)-\nabla_x f(x_{t},y_{t};\zeta_{t+1}^i)\big)\big]\\
=\; & \frac{d_x+2}{d_x^2}\mathbb{E}_{\xi_{t+1}^i}\big[\|\nabla_x f(x_{t+1},y_{t+1};\zeta_{t+1}^i)-\nabla_x f(x_{t},y_{t};\zeta_{t+1}^i)\|^2\big]
\end{align*}
Using the smoothness of $f$, we can further bound $\mathbb{E}[\|a_2\|^2]$ by:
\begin{align*}
\mathbb{E}[\|a_2\|^2]\leq &\frac{45L^2r^2d_x^2}{2q}+ \frac{6L^2(d_x+2)}{q}\mathbb{E} \left[\|x_{t+1}-x_t\|^2+\|y_{t+1}-y_t\|^2\right]\\
\leq &\frac{45L^2r^2d_x^2}{2q}+ \frac{6L^2(d_x+2)\eta_t^2}{q}\mathbb{E}\left[\|\Delta_x^t\|^2+\|\Delta_y^t\|^2\!\right].
\end{align*}
% For any $\xi_{t+1}^i$, we can further control $\|a_2\|^2$ by
% \begin{align*}
% \mathbb{E}[\|a_2\|^2]\leq\;& \frac{1}{q}\left\|\hat{\nabla}_x f(x_{t+1},y_{t+1};\xi_{t+1}^i)-\hat{\nabla}_x f(x_{t},y_{t};\xi_{t+1}^i\right\|^2\\
% \leq\;& \frac{1}{q}\left\|\hat{\nabla}_x f(x_{t+1},y_{t+1};\xi_{t+1}^i)-\nabla_x f(x_{t+1},y_{t+1})+\nabla_x f(x_{t+1},y_{t+1})-\nabla_x f(x_{t},y_{t})\right.\\
% &\qquad +\left.\nabla_x f(x_{t},y_{t})-\hat{\nabla}_x f(x_{t},y_{t};\xi_{t+1}^i)\right\|^2\\
% \leq\;& \frac{3L^2r_1^2d_x}{2q}+\frac{3}{q} \left\|\nabla_x f(x_{t+1},y_{t+1})-\nabla_x f(x_{t},y_{t})\right\|^2\\
% \leq\;& \frac{3L^2r_1^2d_x}{2q}+\frac{3L^2}{q}\left(\|x_{t+1}-x_t\|^2+\|y_{t+1}-y_t\|^2\right)\\
% \leq\;& \frac{3L^2r_1^2d_x}{2q}+\frac{3L\eta_t^2}{q}\left(\|\Tilde{x}_{t+1}-x_t\|^2+\|\Tilde{y}_{t+1}-y_t\|^2\right).
% \end{align*}
Finally, we can derive the bound for $\mathbb{E}\left[\|\Delta_v^{t+1}\|^2\right]$:
$$\mathbb{E}\left[\|\Delta_v^{t+1}\|^2\right]
\leq\; \mathbb{E}\left[(1-\alpha_{t+1})^2 \left\| \Delta_v^t\right\|^2\right]+\frac{2\alpha_{t+1}^2\sigma^2}{q}
+\frac{45L^2r^2d_x^2}{q}+ \frac{12L^2(d_x+2)\eta_t^2}{q}\mathbb{E}
\left[\|\Delta_x^t\|^2+\|\Delta_y^t\|^2\right].$$
% \begin{align*}
% \mathbb{E}\left[\|\Delta_v^{t+1}\|^2\right]
% \leq\; \mathbb{E}\left[(1-\alpha_{t+1})^2 \left\| \Delta_v^t\right\|^2\right]+\frac{2\alpha_{t+1}^2\sigma^2}{q}
% +\frac{45L^2r^2d_x}{q}\!+\! \frac{12L^2(d_x+2)\eta_t^2}{q}\!\mathbb{E}
% \!\left[\!\|\Delta_x^t\|^2\!+\!\|\Delta_y^t\|^2\!\right].
% \end{align*}
Taking one more step, we have
\begin{align*}
&\mathbb{E}\left[\frac{1}{\eta_t}\|\Delta_v^{t+1}\|^2-\frac{1}{\eta_{t-1}}\|\Delta_v^t\|^2\right]\\
\leq\;& \mathbb{E}\left[\left(\frac{(1-\alpha_{t+1})^2}{\eta_t}-\frac{1}{\eta_{t-1}}\right) \left\| a_1\right\|^2\right]+\frac{2\alpha_{t+1}^2\sigma^2}{q\eta_t}+\frac{12L^2(d_x+2)\eta_t}{q}\mathbb{E}\left[\|\Delta_x^t\|^2+\|\Delta_y^t\|^2\right]+\frac{45L^2r^2d_x^2}{q\eta_t}.
\end{align*}
By $\alpha_{t+1}=c\eta_t^2$ with $c\geq 2$, $\eta_t\leq \eta_0\leq\frac{1}{c}$, and the inequality $\frac{1}{\eta_t}-\frac{1}{\eta_{t-1}}\leq \frac{2}{3}\eta_t$ (Eq.(93) in \cite{huang2023adagda}), we have $\frac{(1-\alpha_{t+1})^2}{\eta_t}-\frac{1}{\eta_{t-1}}\leq \frac{2\eta_t}{3}-\frac{2\alpha_{t+1}}{\eta_t}+c^2\eta_t^3\leq -\frac{\alpha_{t+1}}{\eta_t}$. Substituting this inequality into the above one can finish the proof of the first inequality in Lemma \ref{lemma:v_error_bound}. The proof of the second inequality follows the same procedure.

\subsection{Proof of Lemma \ref{lemma:bound_T}}
We try to bound $T_1, T_2, T_3$, and $T_4$ respectively. First, since $\eta_t$ are decreasing with $t$, we have $$T_1 = \frac{16}{T}\sum_{t=1}^T\frac{(\phi_t-\phi_{t+1})}{\gamma\eta_T}
=\frac{16(\phi_1-\phi_{T+1})}{ k_1 T}\cdot (m+T)^{\frac{11-3e}{15-3e}}.$$
% \begin{align*}
% T_1 & \leq \frac{16}{T}\sum_{t=1}^T\frac{(\phi_t-\phi_{t+1})}{\gamma\eta_T}\\
% &=\frac{16(\phi_1-\phi_{T+1})}{ k_1 T}\cdot (m+T)^{\frac{1}{3}+\frac{6-2e}{15-3e}}\\
% &\leq \frac{16(\phi_1-\phi_{T+1})}{ k_1 }\left(\frac{m^{\frac{11-3e}{15-3e}}}{T}+\frac{1}{T^{\frac{4}{15-3e}}}\right).
% \end{align*}
By the definition of $\phi_t$, we can bound $\phi_1$ by
% $$\phi_1\leq F_r(x_1)+\frac{8L^2k_1\mu}{k_2} D_y^2+ \frac{2\gamma\sigma^2}{\mu^2\eta_{0}q}.$$
\begin{align*}
\phi_1 \leq\; & \mathbb{E}\big[F_r(x_1)+\frac{8L^2k_1\mu}{k_2}\|y_{1}-y_r^*(x_{1})\|^2+\frac{\gamma}{\mu^2\eta_{0}}\big( \|\nabla_x f_{1}(x_1,y_1)-v_1\|^2+\|\nabla_y f_{2}(x_1,y_1)-w_1\|^2\big)\big]\\
\leq\; & F_r(x_1)+\frac{8L^2k_1\mu}{k_2} D_y^2+ \frac{2k_1\mu^{1-e}\sigma^2}{\eta_{0}q}.
\end{align*}
It leads to $\phi_1-\phi_{T+1}\leq F_r(x_1)+\frac{8L^2k_1\mu}{k_2} D_y^2+ \frac{2k_1\mu^{1-e}\sigma^2}{\eta_{0}q}-F_r^*$, where $F_r^*=\min_{x\in\mathcal{X}}F_r(x)$.
Denote $R_1=\frac{16}{k_1}\big(F_r(x_1)+\frac{8L^2k_1\mu}{k_2} D_y^2+ \frac{2k_1\mu^{1-e}\sigma^2}{\eta_{0}q}
-F_r^*\big)$. Then, we can bound $T_1$ by
$$T_1\leq \frac{R_1}{ T }(m+T)^{\frac{11-3e}{15-3e}}.$$
% \begin{align*}
% T_1\leq \frac{16M_1}{ k_1 }\left(\frac{m^{\frac{11-3e}{15-3e}}}{T}+\frac{1}{T^{\frac{4}{15-3e}}}\right)=O\left( T^{-\frac{4}{15-3e}}\right).
% \end{align*}
Due to $r\leq\frac{1}{L\sqrt{d}(m+T)^{2/3}}$ and $\eta_t=\frac{1}{(m+t)^{1/3}}$, we have 
$$T_2=\frac{16}{T}\sum_{t=1}^T\frac{192L^4 r^2 \eta_t}{\mu^2\eta_T} <\frac{3200L^4 r^2}{\mu^2T\eta_T^3}\sum_{t=1}^T\eta_t^3
 \leq \frac{3200L^2\ln{(m+T)}}{d T (m+T)^{\frac{1-e}{15-3e}}}\leq \frac{3200L^2\ln{(m+T)}}{d T}.$$
% \begin{align*}
% T_2 & \leq \frac{16}{T}\sum_{t=1}^T\frac{96L^4d_y r^2}{\mu^2}\\
% & \leq \frac{16}{T}\sum_{t=1}^T\frac{96L^2}{ (m+t)(m+T)^{\frac{1-e}{15-3e}}}\\
% &\leq \frac{1600L^2\ln{(m+T)}}{T}.
% \end{align*}
We bound $T_3$ by 
$$T_3 =\frac{16}{T}\sum_{t=1}^T\frac{45L^2d^2 r^2}{ q\eta_t\eta_T\mu^2}\leq \frac{720L^2d^2r^2}{q\mu^2 T}\sum_{t=1}^T(m+T)^{\frac{2}{3}} \leq \frac{720L^2d^2r^2}{q\mu^2 }(m+T)^{\frac{2}{3}}
\leq \frac{720d}{q(m+T)^{\frac{6-2e}{15-3e}}},$$
which leads to $T_3\leq \frac{720d}{q}(m+T)^{-\frac{4}{15-3e}}$.
Finally, we bound $T_4$ by 
$$T_4=\frac{16}{T}\sum_{t=1}^T\frac{4c^2 \eta_t^3\sigma^2}{  q\mu^2\eta_T}=\frac{64c^2\sigma^2}{q\mu^2T\eta_T}\sum_{t=1}^T (m+t)^{-1}  
\leq \frac{64c^2\sigma^2\ln (m+T)}{q\mu^2T\eta_T}=\frac{64c^2\sigma^2\ln(m+T)}{qT}(m+T)^{\frac{9-e}{15-3e}}.$$
% \frac{16}{T}\sum_{t=1}^T\frac{4c^2\eta_t^3\sigma^2}{q\eta_T\mu^{2}}
% \leq \frac{16}{T}\sum_{t=1}^T
% \frac{4c^2\sigma^2 (m+T)^{\frac{9-e}{15-3e}}}{q(m+t)}
% \leq  \frac{64c^2\sigma^2\ln(m+T)}{qT}\left(m+T\right)^{\frac{9-e}{15-3e}}.
Combining the above bounds for $T_1, T_2, T_3,$ and $T_4$ finishes the proof.
% \begin{align*}
% T_4\leq & \frac{16}{T}\sum_{t=1}^T\frac{4c\eta_t^3\sigma^2}{k_1\eta_T\mu^{2-e}}\\
% \leq & \frac{16}{T}\sum_{t=1}^T
% \frac{4c\sigma^2 (m+T)^{\frac{3-e}{5-e}}}{k_1(m+t)}\\
% \leq & \frac{64c\sigma^2\ln(m+T)}{k_1}\left(\frac{m^{\frac{3-e}{5-e}}}{T}+\frac{1}{T^{\frac{2}{5-e}}}\right).
% \end{align*}

\section{Experimental Detail}
\subsection{Serial Queuing System}
\paragraph{Problem Formulation.} We consider a serial queuing system consisting of five sequential subqueues, where a single server serves each subqueue. The goal is to minimize service cost by adjusting the service rates of all servers while considering the service quality. The optimization problem is formulated as
\begin{align*}
\min_{x\in\mathcal{X}}\; & h_0(x)=c^T x\\
\text{s.t.}\; & W(x)=\mathbb{E}[W(x;\xi)] \leq \overline{w},
\end{align*}
where $x\in\mathcal{X}\subset\mathbb{R}^5$ is the vector of service rates. $h_0: \mathbb{R}^5\to \mathbb{R}$ is the cost function and $c\in\mathbb{R}^5$ is the vector of cost coefficients. $W: \mathcal{X} \to \mathbb{R}$ is the expected average waiting time of the queuing system, and $\overline{w}$ is the maximum waiting time. The expected waiting time has no analytical form and can only be evaluated by the simulation outputs.
\paragraph{Comparison Algorithms.} Besides validating the performance of the MGS algorithm, we also compare it with three simulation optimization algorithms that can solve this problem. The first one is a penalty-based gradient search (PGS) algorithm \citep{wang2008stochastic}, where the constraint violation is added to the objective function to be minimized:
\begin{align}\label{eq:penalty_obj}
    \overline{h}_0(x)= h_0(x)+ w_p\left(\max( W(x)-\overline{w}, 0) \right)^2.
\end{align}
Wherein, the positive scalar $w_p$ is the parameter to adjust the weight on the constraint violation. Then, similar to (6) and (7) in the main body of our paper, we use $\bar{q}$ two-point gradient estimators to construct averaged gradients of \eqref{eq:penalty_obj} to perform gradient descent. Specifically, the gradient of \eqref{eq:penalty_obj} can be estimated by
\begin{align*}
\hat{\nabla}_x \overline{h}_0(x;\mathcal{S})=\frac{1}{\overline{q}}\sum_{i=1}^{\overline{q}}\frac{d_x(\overline{h}_0(x+rz^i;\xi^i)-\overline{h}_0(x;\xi^i))}{r}\cdot z.
\end{align*}
The second one is the Bayesian optimization (BO) based on constrained expected improvement (CEI), which is widely applied to solve constrained simulation optimization problems \citep{gardner2014bayesian,ragueneau2024constrained}. While the computational complexity of CEI is too high to handle high-dimensional problems, it can typically generate high-quality solutions for problems with a small feasible space. Similarly, function evaluations are required for the sampling of CEI. We apply $20$ repeated simulations to evaluate the average total waiting time for each sampled point. The third one is the genetic algorithm (GA), which is widely applied to solve complex problems. While the computational complexity of GA is too high to handle high-dimensional problems, it can typically generate high-quality solutions for problems with a small feasible space. Due to its prevalence \citep{rajwar2023exhaustive}, we omit a detailed description of the algorithm here. For GA, we also use $q$ replications to estimate the average waiting time for each $x$ during the solution process.

% The second one is the genetic algorithm (GA), which is widely applied to solve complex problems. While the computational complexity of GA is too high to handle high-dimensional problems, it can typically generate high-quality solutions for problems with a small feasible space. Due to its prevalence \citep{rajwar2023exhaustive}, we omit a detailed description of the algorithm here. For GA, we also use $q$ replications to estimate the average waiting time for each $x$ during the solution process.
\paragraph{Parameters Setting.} 
The detailed parameters of the queuing system are listed in Table \ref{tab:env_parameter}.
% The feasible space is $\mathcal{X}=[1,5]^5$; the maximum waiting time is set as $\overline{w}=5$; the cost coefficient is set as $c=(10,6,6,8,10)^T$. For each simulation, the arrival rate is set as 1, and the total number of customers is 1000. 
\begin{table}[htpb]
\centering
\caption{Parameter settings of the queuing system}
\small
\begin{tabular}{cc}
\hline
Parameter                   & Value              \\ \hline
$\mathcal{X}$               & $[1,5]^5$          \\
Max. waiting time $\Bar{w}$ & 5                  \\
Arrival rate                & 1                  \\
Cost coefficient $c$        & $(10, 6,6,8,10)^T$ \\
Number of customers         & 1000               \\ \hline
\end{tabular}
\label{tab:env_parameter}
\end{table}
We test all algorithms with multiple parameter settings and choose the ones with the best performance. All algorithms are tested with 300 iterations. The smoothing radius is set as $r=10^{-3}$ for all gradient estimations. For MGS, we run the algorithms with $q=1,10,20, 50$, respectively. We set $\mathcal{Y}=[0,1000]$, $\mu=10^{-3}$, $\eta_t=(10+t)^{-\frac{1}{3}}$, and $\alpha_t=\beta_t=6\eta_t^2$ for all MGS algorithms.
We set $\gamma=0.01,\; \lambda=0.1$ when $q=1$, and $\gamma=0.05,\; \lambda=0.2$ when $q=10, 20, 50$. For PGS, we set the weight parameter as $w_p=10$ and use diminishing step sizes, i.e., $\eta_t=0.02/\sqrt{t+1}$. Moreover, for each gradient estimation, we perform $\overline{q}=20$ replications to take the average value. For CEI, the number of initial samples is set to 40, and the total number of samples is 300. For GA, the population size is 30 with 50 generations. The mutation rate is 0.15, and the crossover rate is 0.85. To estimate the average waiting time of each iterate, we independently repeat the simulations 20 times and calculate the average value.

% For GA, the population size is 30 with 50 generations. The mutation rate is 0.15, and the crossover rate is 0.85. To estimate the average waiting time and stationarity measure of each iterate, we independently sample the random variables 20 times to repeat the simulation.

\subsection{High-Dimensional Case}
\paragraph{Problem Formulation.} We consider a high-dimensional case to verify the capability of MGS to deal with high-dimensional problems. Specifically, we consider the stochastic optimization problem:
\begin{equation}
\begin{aligned}\label{eq:high_dimensional_problem}
\min_{x\in\mathcal{X}\subseteq \mathbb{R}^{2000}}\; & h_0(x)=\mathbb{E}[h_0(x;a)]=\mathbb{E}[\langle a, x^3\rangle] -5x^Tx\\
 \text{s.t.}\;& h_1(x)=\mathbb{E}[h_1(x;b)]=\mathbb{E}[\langle b,x^2\rangle]-200 \leq 0,
\end{aligned}
\end{equation}
where $a,b\in\mathbb{R}^{2000}$ are random variables following Gaussian distributions. Assume we can only get the noisy function values of the objective and constraint.
\paragraph{Comparison Algorithms.} Due to the inefficiency of CEI to deal with high-dimensional problems, we only compare MGS with PGS for solving problem \eqref{eq:high_dimensional_problem}.

\paragraph{Parameters Setting.} We set the feasible space as $\mathcal{X}=[0,3]^{2000}$. The random variables $a,b$ follow Gaussian distributions $\mathcal{N}(\overline{a},0.5\cdot I_{2000})$ and $\mathcal{N}(\mathbf{1}, 0.05\cdot I_{2000})$, where $\overline{a}$ is a constant vector randomly generated with all entries centered at 2. For MGS, we test it with $q=5,10$, respectively. We set $\mathcal{Y}=[0,1000]$, $\mu=10^{-3}$, $\eta_t=(100+t)^{-\frac{1}{3}}$ and $\alpha_t=\beta_t=6\eta_t^2$. Furthermore, we set $\gamma=0.1, \lambda=0.01$ when $q=5$ and $\gamma=0.1, \lambda=0.008$ when $q=10$. For PGS, we set the weight parameter as $w_p=5$ and use diminishing step sizes, i.e., $\eta_t=0.01/\sqrt{t+1}$. Moreover, for each gradient estimation, we perform $\overline{q}=20$ replications to take the average value.

\end{document}